\documentclass[11pt,oneside,reqno]{amsart}

\usepackage{amsmath,amssymb,amsthm,textcomp,mathtools}
\usepackage{amsfonts,graphicx}
\usepackage[mathscr]{eucal}
\usepackage{color}
\usepackage{csquotes}
\usepackage{enumitem}
\allowdisplaybreaks

\numberwithin{equation}{section}

\theoremstyle{definition}

\numberwithin{equation}{section}

\newtheorem{theorem}{\bf Theorem}[section]
\newtheorem{remark}{\bf Remark}[section]

\newtheoremstyle
{remarkstyle}
{}
{11pt}
{}
{}
{\bfseries}
{:}
{     }
{\thmname{#1} \thmnumber{#2} }

\theoremstyle{remarkstyle}

\begin{document}
	\title{On a Fractional Variant of Linear Birth-Death Process}
		\author[Manisha Dhillon]{Manisha Dhillon}
	\address{M. Dhillon, Department of Mathematics,
		Indian Institute of Technology Bhilai, Durg, 491002, India.}
	\email{manishadh@iitbhilai.ac.in}
	\author[Pradeep Vishwakarma]{Pradeep Vishwakarma}
	\address{P. Vishwakarma, Department of Mathematics,
		Indian Institute of Technology Bhilai, Durg, 491002, India.}
	\email{pradeepv@iitbhilai.ac.in}
	\author[Kuldeep Kumar Kataria]{Kuldeep Kumar Kataria}
	\address{K.K. Kataria, Department of Mathematics, Indian Institute of Technology Bhilai, Durg, 491002, India.}
	 \email{kuldeepk@iitbhilai.ac.in}	
	
	
	\date{\today}	
	\maketitle
\begin{abstract}
We introduce and study a fractional variant of the linear birth-death process, namely, the generalized fractional linear birth-death process (GFLBDP) which is defined by taking the regularized Hilfer-Prabhakar derivative in the system of differential equations that governs the state probabilities of linear birth-death process. For a particular choice of parameters, the GFLBDP reduces to the fractional linear birth-death process that involves the Caputo derivative. Its time-changed representation is obtained and utilized to derive the explicit expressions of its state probabilities. The explicit expressions for its mean and variance are derived. In a particular case, it is observed that the limiting distribution of the time changing process coincides to that of an inverse stable subordinator. A relation between the extinction probability of GFLBDP and the density of inter arrival times of a generalized fractional Poisson process is obtained. Later, we study some integrals of the GFLBDP and discuss the asymptotic distributional characteristics for a particular integral process. Also, an application of the path integral at random time to a genetic population with an upper bound is discussed.
\end{abstract}
	
\section{Introduction}
The birth-death process is a continuous-time Markov process which is used to model the growth of population over time. In this process, at any given instant, we can have at most one birth or one death and the possibility of simultaneous birth and death is negligible. In \cite{Vishwakarma2024a}, the birth-death process is generalized to the case of finitely many multiple births and deaths. 

Let $\lambda_n=n\lambda$ and $\mu_n=n\mu$ be the birth and death rates, respectively. In this case, it is referred as the linear birth-death process and we denote it by $\{\mathcal{N}(t)\}_{t\ge0}$. Its state probabilities $p(n,t)=\mathrm{Pr}\{\mathcal{N}(t)=n\}$, $n\ge0$ solve the following system of differential equations (see \cite{Feller1968}):
\begin{equation}\label{bdpeq}
	\mathrm{d}_{t}p(n,t)=-n(\lambda+\mu)p(n,t)+(n-1)\lambda p(n-1,t)+(n+1)\mu p(n+1,t),
\end{equation}
with $p(1,0)=1$,
where $\mathrm{d}_t$ denotes the derivative with respect to $t$. Its extinction probability is given by (see \cite{Bailey1964}, p. 93)
\begin{equation}\label{lbdppgf}
	p(0,t)=\begin{cases}
		\frac{\lambda t}{1+\lambda t},\ \lambda=\mu,\vspace{0.2cm}\\
		\frac{\mu-\mu e^{(\lambda-\mu)t}}{\mu-\lambda e^{(\lambda-\mu)t}},\ \lambda\ne\mu
	\end{cases}
\end{equation}
and for $n\ge1$, its state probabilities are
\begin{equation}\label{bdpsp}
	p(n,t)=\begin{cases}
		\frac{(\lambda t)^{k-1}}{(1+\lambda t)^{k+1}},\ \lambda=\mu,\vspace{0.2cm}\\
		(\lambda-\mu)^2e^{-(\lambda-\mu)t}\frac{(\lambda-\lambda e^{-(\lambda-\mu)t})^{k-1}}{(\lambda-\mu e^{-(\lambda-\mu)t})^{k+1}},\ \lambda\neq\mu.
	\end{cases}
\end{equation} 

In the past few years, many time-changed and time-fractional growth processes have been studied, for example, the fractional Poisson process (see \cite{Beghin2009}, \cite{Leskin2003}, \cite{Meerschaert2011}), the fractional pure birth process (see \cite{Orsingher2010}), \textit{etc}. An important characteristic of these fractional processes is their global memory property, which is essential for modeling many real world systems.
Orsingher and Polito \cite{Orsingher2011} introduced and studied a fractional variant of the linear birth-death process, where they replaced the integer order derivative in \eqref{bdpeq} by the Caputo fractional derivative. Recently, Kataria and Vishwakarma \cite{Kataria2024c} studied the fractional variant of the birth-death process for arbitrary birth and death rates. They obtained the closed form expressions for its state probabilities using the Adomian decomposition method. Also, the asymptotic behaviour of the distribution function of the extinction time is analyzed. The fractional variants of linear birth-death process have faster mean growth compared to its non-fractional case. Therefore, it serves as a more effective mathematical model for systems governed by a birth-death process under rapidly changing conditions. For more details on the time-changed variants of birth-death processes, we refer the reader to \cite{Kataria2024d} and \cite{Vishwakarma2024b}.

Garra \textit{et al.} \cite{Garra2014} introduced and studied a fractional generalization of the homogeneous Poisson process where the governing system of difference-differential equations of its distribution involves the regularized Hilfer-Prabhakar derivative. Recently, Khandakar \textit{et al.} \cite{Khandakar2025}
did a similar study for the generalized counting process, birth and death processes by replacing the integer order derivative involved in the governing system of differential equations of their state probabilities with a regularized Hilfer-Prabhakar derivative. 

In this paper, we define and study a generalized fractional linear birth-death process (GFLBDP) by replacing the integer order derivative in (\ref{bdpeq}) with regularized Hilfer-Prabhakar derivative. We establish a time-changed representation of GFLBDP that connects it with the linear birth-death process via a random process whose density solves a fractional Cauchy problem. We introduce some integrals of the GFLBDP and study the limiting behaviour for a particular case. It is observed that the fractional linear birth-death process studied in \cite{Orsingher2011} is a particular case of the GFLBDP. 

The GFLBDP can be used to model real-life situations that involve multiple parameters. In Section \ref{sec2}, it is shown that the GFLBDP exhibits faster mean growth compared to the homogeneous linear birth-death process with respect to one parameter, and grows more slowly with respect to another parameter. So, it is a suitable mathematical model for analyzing real-world systems where the birth-death processes are applicable.

Let $\mathcal{Y}(t)=\int_{0}^{t}\mathcal{N}(s)\,\mathrm{d}s$ be the path integral of linear birth-death process $\{\mathcal{N}(t)\}_{t \geq 0}$. Then, the bivariate process $\{(\mathcal{N}(t),\mathcal{Y}(t))\}_{t\ge0}$ naturally arises in several fields, such as biological sciences, engineering, inventory systems, \textit{etc.} For example, it can be applied to study the lethal effects of bacteria. Here, $\mathcal{N}(t)$ represents the total number of bacteria inside the host body at any time $t$, and $\mathcal{Y}(t)$ is proportional to the cumulative amount of toxin produced by it up to that time. It is empirically observed that the lethal effect of bacteria depends on the total bacterial count and the total toxin produced. Thus, a bivariate time-changed process serves as a potential model to study the effect of bacteria at random time.

The paper is organized as follows:

In Section \ref{sec2}, we give a generalization of the fractional linear birth-death process. It is introduced using a system of fractional differential equations that involves the regularized Hilfer-Prabhakar derivative. Its time-changed representation is derived where the density of time changing component solves a fractional Cauchy problem. By using its time-change representation, its extinction probability is derived for three different cases of birth and death rates, and its asymptotic behaviour is analyzed.

In Section \ref{sec3}, we obtain the explicit expressions for the state probabilities of GFLBDP in three different cases of birth and death rates. Also, we establish that in the long run of the process these state probabilities coincide with that of the fractional linear birth-death process studied in \cite{Orsingher2011}.

In the last section, we study some integrals of the GFLBDP, in particular, we consider its Prabhakar integral. We study the joint distributional properties of the linear birth-death process and its path integral at random time where the density of time changing process solves a fractional Cauchy problem. For particular choices of the  parameters, these integrals reduces to that of the linear birth-death process studied in \cite{Vishwakarma2024b}. Also, we discuss an application of the path integral of a time-changed linear birth-death process to a genetic population with an upper bound.

\section{Generalized fractional linear birth-death process}\label{sec2}
In this section, we define a fractional variant of the linear birth-death process whose state probabilities  solve a system of fractional differential equations involving regularized Hilfer-Prabhakar derivative. We call it the  generalized fractional linear birth-death process (GFLBDP) and denote it by $\{\mathcal{N}_{hp}(t)\}_{t\ge0}$. It is defined as a birth-death process whose state probabilities $p_{hp}(n,t)=\mathrm{Pr}\{\mathcal{N}_{hp}(t)=n\}$, $n\ge0$ solve the following system of differential equations:
\begin{equation}\label{eqdef}
	{}^{hp}\mathcal{D}_{\alpha, -\beta}^{\gamma,\rho}p_{hp}(n,t)=-n(\lambda+\mu)p_{hp}(n,t)+(n-1)\lambda p_{hp}(n-1,t)+(n+1)\mu p_{hp}(n+1,t),
\end{equation}
with initial condition $p_{hp}(1,0)=1$. Also, $p_{hp}(-1,t)=0$ for all $t\ge0$ and $|\rho\lceil\gamma\rceil/\gamma-j\alpha|<1$ for all $j=0,1,\dots,\lceil\gamma\rceil$ whenever $\gamma\ne0$, where $\lceil\cdot\rceil$ denotes the ceiling function. 

In (\ref{eqdef}), the operator ${}^{hp}\mathcal{D}_{\alpha, -\beta}^{\gamma,\rho}$, $0<\alpha\leq1$, $\beta>0$, $\gamma\ge0$, $0<\rho\leq1$ is the regularized Hilfer-Prabhakar derivative defined as follows (see \cite{Garra2014}):
\begin{equation}\label{fder}
	{}^{hp}\mathcal{D}_{\alpha, -\beta}^{\gamma,\rho}f(t)=(\textbf{E}_{\alpha,1-\rho,\beta}^{-\gamma}\mathrm{d}_xf(x))(t),
\end{equation}
for all absolutely continuous function $f(\cdot)$. Here, $\textbf{E}_{\alpha,\rho,\beta}^{\gamma}$ is the Prabhakar integral defined by (see \cite{Prabhakar1971})
\begin{equation}\label{fi}
	\textbf{E}_{\alpha,\rho,\beta}^{\gamma}g(t)=\int_{0}^{t}(t-s)^{\rho-1}E_{\alpha,\rho}^\gamma(\beta(t-s)^\alpha)g(s)\,\mathrm{d}s,\ \alpha>0,\ \rho>0,\ \beta\in\mathbb{R},\, \gamma\in\mathbb{R},
\end{equation}
for all integrable function $g(\cdot)$, where $E_{\alpha,\beta}^\gamma(\cdot)$ is the three parameter Mittag-Leffler function defined as (see \cite{Kilbas2006})
\begin{equation}\label{Mittag3}
	E_{\alpha,\beta}^\gamma(x)\coloneqq\sum_{k=0}^{\infty}\frac{(\gamma)_kx^k}{\Gamma(k\alpha+\beta)k!},\ x\in\mathbb{R}.
\end{equation}

For $\gamma=0$, the integral (\ref{fi}) reduces to the following Riemann-Liouville fractional integral (see \cite{Kilbas2006}):
\begin{equation*}
	I_{\rho}g(t)=\frac{1}{\Gamma(\rho)}\int_{0}^{t}(t-s)^{\rho-1}g(s)\,\mathrm{d}s.
\end{equation*}
Thus, for $\gamma=0$, the system of fractional differential equations (\ref{eqdef}) reduces to that of fractional linear birth-death process introduced and studied by Orsingher and Polito \cite{Orsingher2011}.

Next, we derive a time-changed representation of the GFLBDP. Let $\{\mathscr{Q}(t)\}_{t\ge0}$ be a random process whose density $f_{\mathscr{Q}}(x,t)=\mathrm{Pr}\{\mathscr{Q}(t)\in\mathrm{d}x\}/\mathrm{d}x$, $x\ge0$ solves the following Cauchy problem:
\begin{equation}\label{tc}
	{}^{hp}\mathcal{D}_{\alpha, -\beta}^{\gamma,\rho}f_{\mathscr{Q}}(x,t)=-\partial_xf_{\mathscr{Q}}(x,t),
\end{equation}
with $f_{\mathscr{Q}}(x,0)=\delta(x)$, where $\delta$ is the Dirac delta function.
The double Laplace transform of $f_{\mathscr{Q}}(x,t)$ is given by (see \cite{Garra2014}, Eq. (67))
\begin{equation}\label{tc2lap}
	\int_{0}^{\infty}\int_{0}^{\infty}e^{-zx-wt}f_{\mathscr{Q}}(x,t)\,\mathrm{d}x\,\mathrm{d}t=\frac{w^{\rho-1}(1+\beta w^{-\alpha})^\gamma}{w^{\rho}(1+\beta w^{-\alpha})^\gamma+z},\ w>0,\ z>0.
\end{equation}
So, its Laplace transform with respect to time variable is 
\begin{equation}\label{tclap}
	\int_{0}^{\infty}e^{-wt}f_{\mathscr{Q}}(x,t)\,\mathrm{d}t=w^{\rho-1}(1+\beta w^{-\alpha})^\gamma\exp(-w^{\rho}(1+\beta w^{-\alpha})^\gamma x),\ x\ge0,\ w>0.
\end{equation}

Now, we recall the definition of a stable subordinator and its inverse. A non-decreasing L\'evy process $\{L_\nu(t)\}_{t\ge0}$, $0<\nu<1$ is called the $\nu$-stable subordinator if its Laplace transform is $\mathbb{E}e^{-zL_\nu(t)}=e^{-tz^\nu}$, $z>0$. Its first passage time process $\{\mathcal{L}_\nu(t)\}_{t\ge0}$ given by $\mathcal{L}_\nu(t)\coloneqq\inf\{u>0:L_\nu(u)>t\}$, $t\ge0$ is called the inverse $\nu$-stable subordinator. It has the following Laplace transform (see \cite{Meerschaert2011}):
\begin{equation}\label{insublap}
	\mathbb{E}e^{-z\mathcal{L}_\nu(t)}=E_{\nu,1}(-zt^\nu),\ z>0.
\end{equation}

Let $\{L_{{\gamma}/{\lceil\gamma\rceil}}(t)\}_{t\ge0}$ be a ${\gamma}/{\lceil\gamma\rceil}$-stable subordinator and $Y_j(t)\coloneqq\binom{\lceil\gamma\rceil}{j}L_{{\gamma}/{\lceil\gamma\rceil}}(t)$, $j\ge0$. Let us consider the following time-changed random process:
\begin{equation}\label{tcdef}
	\mathscr{B}(t)\coloneqq\sum_{j=0}^{\lceil\gamma\rceil}L_{\rho\frac{\lceil\gamma\rceil}{\gamma}-j\alpha}(Y_j(t)),\ t\ge0,
\end{equation} 
with $0<\rho{\lceil\gamma\rceil}/{\gamma}-j\alpha<1$ for all $j=0,1,\dots,\lceil\gamma\rceil$. In (\ref{tcdef}), it is assumed that all the component processes are independent of each other. The first passage time process of $\{\mathscr{B}(t)\}_{t\ge0}$  is equal in distribution to  $\{\mathscr{Q}(t)\}_{t\ge0}$ (see \cite{Garra2014}), that is,
\begin{equation*}
	\mathscr{Q}(t)\overset{d}{=}\inf\{u\ge0:\mathscr{B}(u)>t\},\ t\ge0,
\end{equation*}
where $\overset{d}{=}$ denotes equality in distribution.

\begin{remark}\label{asymprem}
		If $|zw^{-\rho}(1+\beta w^{-\alpha})^{-\gamma}|<1$, $w>0$ then (\ref{tc2lap}) can be rewritten as follows:
		\begin{align*}
			\int_{0}^{\infty}\int_{0}^{\infty}e^{-zx-wt}f_{\mathscr{Q}}(x,t)\,\mathrm{d}x\,\mathrm{d}t&=\frac{1}{w}\bigg(1+\frac{z}{w^{\rho}(1+\beta w^{-\alpha})^\gamma}\bigg)^{-1}\\
			&=\sum_{k=0}^{\infty}(-z)^k{w^{-k\rho-1}(1+\beta w^{-\alpha})^{-k\gamma}}.
		\end{align*}
		On using the following result (see \cite{Kilbas2006}):
		\begin{equation}\label{mllap}
			\int_{0}^{\infty}e^{-wt}t^{\beta-1}E_{\alpha,\beta}^{\gamma}(ct)\,\mathrm{d}t=\frac{w^{\alpha\gamma-\beta}}{(w^{\alpha}-c)^\gamma},\ |cw^{-\alpha}|<1,
		\end{equation}
		we get
		\begin{equation*}
			\int_{0}^{\infty}e^{-zx}f_{\mathscr{Q}}(x,t)\,\mathrm{d}x=\sum_{k=0}^{\infty}(-zt^\rho)^kE_{\alpha,k\rho+1}^{k\gamma}(-\beta t^\alpha).
		\end{equation*}
		Now, we recall that for large $t$, we have the following limiting result (see \cite{Beghin2012}, Eq. (2.44)):
		\begin{equation}\label{mittaglim}
			E_{\alpha,\beta}^{\gamma}(-ct^\alpha)\sim\frac{(ct^{\alpha})^{-\gamma}}{\Gamma(\beta-\alpha\gamma)},\ \beta\neq\alpha\gamma.
		\end{equation}
		So, if $\rho>\alpha\gamma$ then for sufficiently large $t$, we have 
		\begin{equation}\label{tcasymlap}
			\int_{0}^{\infty}e^{-zx}f_{\mathscr{Q}}(x,t)\,\mathrm{d}x\sim\sum_{k=0}^{\infty}(-zt^\rho)^k\frac{(\beta t^\alpha)^{-k\gamma}}{\Gamma(k(\rho-\alpha\gamma)+1)}=E_{\rho-\alpha\gamma,1}(-z\beta^{-\gamma} t^{\rho-\alpha\gamma}),\ z>0.
		\end{equation}
		Thus, if $0<\rho-\alpha\gamma<1$ and $\beta=1$ then from (\ref{insublap}) and (\ref{tcasymlap}), it follows that the limiting distribution of the process $\{\mathscr{Q}(t)\}_{t\ge0}$ coincides with the distribution of an inverse stable subordinator with index $\rho-\alpha\gamma$.
\end{remark}

\begin{theorem}\label{thmtc}
	Let $\{\mathcal{N}(t)\}_{t\ge0}$ be the homogeneous birth-death process which is mutually independent of a random process $\{\mathscr{Q}(t)\}_{t\ge0}$ whose density solves (\ref{tc}). Then, the GFLBDP satisfies the following time-changed representation:
	\begin{equation}\label{tcrep}
		\mathcal{N}_{hp}(t)\overset{d}{=}\mathcal{N}(\mathscr{Q}(t)),\ t\ge0.
	\end{equation}	
\end{theorem}
\begin{proof}
	For $n\ge0$, we have $\mathrm{Pr}\{\mathcal{N}(\mathscr{Q}(t))=n\}=\int_{0}^{\infty}\mathrm{Pr}\{\mathcal{N}(x)=n\}\mathrm{Pr}\{\mathscr{Q}(t)\in\mathrm{d}x\}$. On using (\ref{tclap}), its Laplace transform with respect to time variable is given by
	\begin{equation}\label{pf1}
		\int_{0}^{\infty}e^{-wt}\mathrm{Pr}\{\mathcal{N}(\mathscr{Q}(t))=n\}\,\mathrm{d}t=w^{\rho-1}(1+\beta w^{-\alpha})^\gamma\int_{0}^{\infty}p(n,x)\exp(-w^{\rho}(1+\beta w^{-\alpha})^{\gamma}x)\,\mathrm{d}x.
	\end{equation}
	By taking the Laplace transform on both sides of (\ref{bdpeq}) with respect to time variable and using (\ref{tclap}), we have
	\begin{align}
		w\int_{0}^{\infty}&e^{-wt}p(n,t)\,\mathrm{d}t-p(n,0)\nonumber\\
		&=\int_{0}^{\infty}e^{-wt}(-n(\lambda+\mu)p(n,t)+(n-1)\lambda p(n-1,t)+(n+1)\mu p(n+1,t))\,\mathrm{d}t.\label{pf2}
	\end{align}
	By replacing $w$ with $w^{\rho}(1+\beta w^{-\alpha})^{\gamma}$ in (\ref{pf2}), we get
	\begin{align}
		w^{\rho}(1+\beta w^{-\alpha})^{\gamma}\int_{0}^{\infty}&e^{-w^{\rho}(1+\beta w^{-\alpha})^{\gamma}t}p(n,t)\,\mathrm{d}t-p(n,0)\nonumber\\
		&=\int_{0}^{\infty}e^{-w^{\rho}(1+\beta w^{-\alpha})^{\gamma}t}(-n(\lambda+\mu)p(n,t)+(n-1)\lambda p(n-1,t)\nonumber\\
		&\hspace{7cm}+(n+1)\mu p(n+1,t))\,\mathrm{d}t.\label{pf3}
	\end{align}
	Now, on multiplying $w^{\rho-1}(1+\beta w^{-\alpha})^\gamma$ on both sides of (\ref{pf3}), and using (\ref{pf1}) and $\mathrm{Pr}\{\mathcal{N}(\mathscr{Q}(0))=n\}=p(n,0)$, we get
	\begin{align}
		w^{\rho}&(1+\beta w^{-\alpha})^\gamma\int_{0}^{\infty}e^{-wt}\mathrm{Pr}\{\mathcal{N}(\mathscr{Q}(t))=n\}\,\mathrm{d}t-w^{\rho-1}(1+\beta w^{-\alpha})^\gamma\mathrm{Pr}\{\mathcal{N}(\mathscr{Q}(0))=n\}\nonumber\\
		&\hspace{1.5cm}=\int_{0}^{\infty}e^{-wt}(-n(\lambda+\mu)\mathrm{Pr}\{\mathcal{N}(\mathscr{Q}(t))=n\}+(n-1)\lambda \mathrm{Pr}\{\mathcal{N}(\mathscr{Q}(t))=n-1\}\nonumber\\
		&\hspace{8.2cm}+(n+1)\mu\mathrm{Pr}\{\mathcal{N}(\mathscr{Q}(t))=n+1\})\,\mathrm{d}t.\label{pf4}
	\end{align}
	By taking the Laplace transform on both sides of (\ref{eqdef}) and using the following result (see \cite{Garra2014}):
	\begin{equation}\label{fderlap}
		\int_{0}^{\infty}e^{-wt}{}^{hp}\mathcal{D}_{\alpha, -\beta}^{\gamma,\rho}f(t)\,\mathrm{d}t=w^{\rho}(1+\beta w^{-\alpha})^\gamma\int_{0}^{\infty}e^{-wt}f(t)\,\mathrm{d}t-w^{\rho-1}(1+\beta w^{-\alpha})^\gamma f(0),
	\end{equation}
	we get
	\begin{align}
		w^{\rho}&(1+\beta w^{-\alpha})^\gamma\int_{0}^{\infty}e^{-wt}p_{hp}(n,t)\,\mathrm{d}t-w^{\rho-1}(1+\beta w^{-\alpha})^\gamma p_{hp}(n,0)\nonumber\\
		&\hspace{0.5cm}=\int_{0}^{\infty}e^{-wt}(-n(\lambda+\mu)p_{hp}(n,t)+(n-1)\lambda p_{hp}(n-1,t)+(n+1)\mu p_{hp}(n+1,t))\,\mathrm{d}t.\label{pf5}
	\end{align}
	From (\ref{pf4}) and (\ref{pf5}), and by using the uniqueness  of distribution function, we conclude that $p_{hp}(n,t)=\mathrm{Pr}\{\mathcal{N}(\mathscr{Q}(t))=n\}$, $n\ge0$. This completes the proof.
\end{proof}

From (\ref{eqdef}), the probability generating function (pgf) $\mathcal{G}_{hp}(u,t)\coloneqq\sum_{n=0}^{\infty}u^np_{hp}(n,t)$, $|u|\leq1$, $t\ge0$ of GFLBDP solves the following fractional differential equation:
\begin{equation}\label{eqpgf}
	{}^{hp}\mathcal{D}_{\alpha, -\beta}^{\gamma,\rho}\mathcal{G}_{hp}(u,t)=(\lambda u-\mu)(u-1)\mathcal{G}_{hp}(u,t),
\end{equation}
with initial condition $\mathcal{G}_{hp}(u,0)=u$.

On taking the derivative with respect to $u$ on both sides of (\ref{eqpgf}) and substituting $u=1$, we get the following governing equation for the mean $\mathbb{E}\mathcal{N}_{hp}(t)=\partial_u\mathcal{G}_{hp}(u,t)|_{u=1}$ of GFLBDP:
\begin{equation}\label{eqmean}
	{}^{hp}\mathcal{D}_{\alpha, -\beta}^{\gamma,\rho}\mathbb{E}\mathcal{N}_{hp}(t)=(\lambda -\mu)\mathbb{E}\mathcal{N}_{hp}(t),
\end{equation}
with $\mathbb{E}\mathcal{N}_{hp}(0)=1$.
On taking the Laplace transform on both sides of (\ref{eqmean}) and using (\ref{fderlap}), we get
\begin{align}
	\int_{0}^{\infty}e^{-wt}\mathbb{E}\mathcal{N}_{hp}(t)\,\mathrm{d}t&=\frac{w^{\rho-1}(1+\beta w^{-\alpha})^\gamma}{w^{\rho}(1+\beta w^{-\alpha})^\gamma-(\lambda-\mu)}\nonumber\\
	&=w^{-1}\bigg(1-\frac{\lambda-\mu}{w^\rho(1+\beta w^{-\alpha})^\gamma}\bigg)^{-1},\ \ \bigg|\frac{\lambda-\mu}{w^\rho(1+\beta w^{-\alpha})^\gamma}\bigg|<1,\nonumber\\
	&=\sum_{k=0}^{\infty}(\lambda-\mu)^kw^{-k\rho-1}(1+\beta w^{-\alpha})^{-k\gamma}.\label{mlap}
\end{align}
Its inversion yields
\begin{equation}\label{meannhp}
	\mathbb{E}\mathcal{N}_{hp}(t)=\sum_{k=0}^{\infty}(\lambda-\mu)^kt^{k\rho}E_{\alpha,k\rho+1}^{k\gamma}(-\beta t^\alpha),\ t\ge0,
\end{equation}
where $E_{\alpha,\beta}^\gamma(\cdot)$ is the three parameter Mittag-Leffler function as defined in (\ref{Mittag3}).

The variation in the mean growth of GFLBDP with time is illustrated in Figure \ref{fig1} and Figure \ref{fig2} for different values of $\rho$ and $\beta$, respectively. From Figure \ref{fig1}, it is observed that for smaller values of $\rho$, we have faster mean growth. In Figure \ref{fig2}, we observe that the mean growth is higher for smaller values of $\beta$.  However, in this case the increase in mean growth is lower than that for $\rho$.
\begin{figure}[ht]
	\includegraphics[width=10cm]{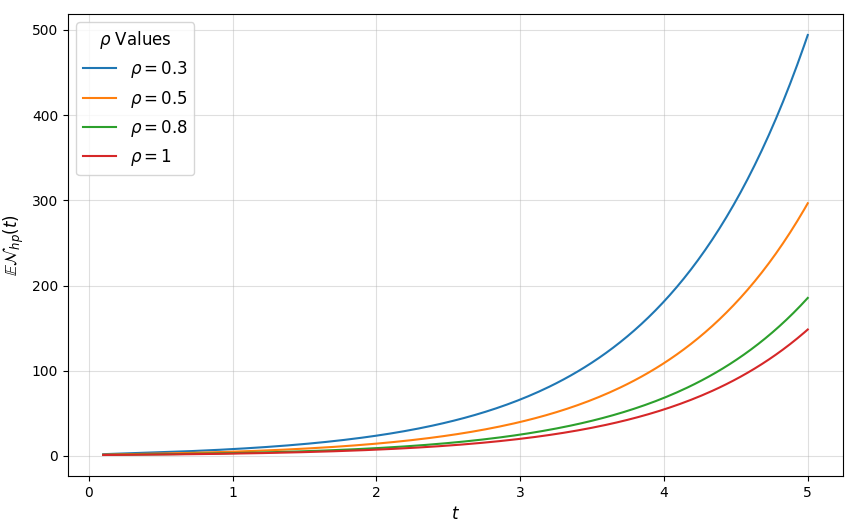}
	\caption{\small Expected value of GFLBDP versus time for different values of $\rho$ where $\lambda-\mu=1$, $\gamma=0$, $\alpha=0.5$ and $\beta=0.5$.}\label{fig1}
\end{figure}
\begin{figure}[ht]
	\includegraphics[width=10cm]{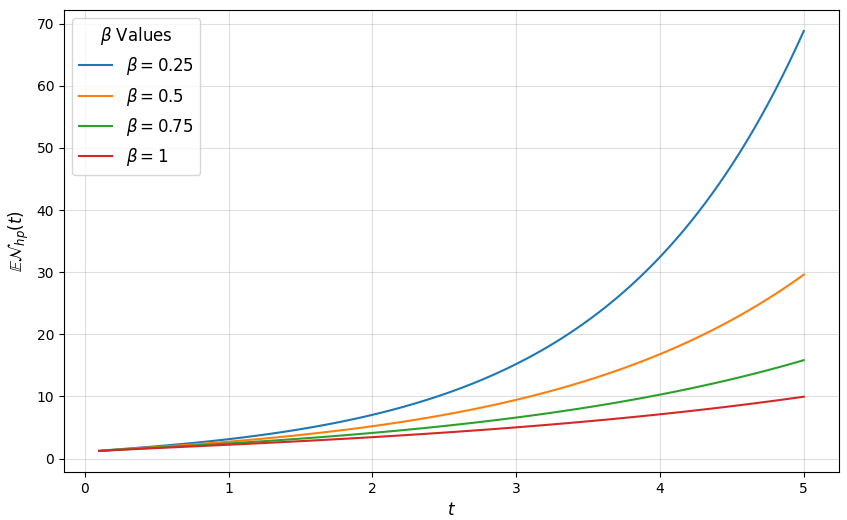}
	\caption{\small Expected value of GFLBDP versus time for different values of $\beta$ where $\lambda-\mu=1$, $\gamma=0.8$, $\alpha=0.5$ and $\rho=0.7$.}\label{fig2}
\end{figure}

Similarly, the second factorial moment $m_2(t)=\mathbb{E}\mathcal{N}_{hp}(t)(\mathcal{N}_{hp}(t)-1)=\partial_u^2\mathcal{G}_{hp}(u,t)|_{u=1}$ solves 
\begin{equation}\label{sfm}
{}^{hp}\mathcal{D}_{\alpha, -\beta}^{\gamma,\rho}m_2(t)=2\lambda\mathbb{E}\mathcal{N}_{hp}(t)+2(\lambda -\mu)m_2(t),
\end{equation}
with $m_2(0)=0$. On taking the Laplace transform on both sides of (\ref{sfm}), and using (\ref{fderlap}) and (\ref{mlap}), we have
\begin{align*}
	&\int_{0}^{\infty}e^{-wt}m_2(t)\,\mathrm{d}t\\
	&=2\lambda\sum_{k=0}^{\infty}(\lambda-\mu)^k\frac{w^{-k\rho-1}(1+\beta w^{-\alpha})^{-k\gamma}}{(w^{\rho}(1+\beta w^{-\alpha})^\gamma-2(\lambda-\mu))}\\
	&=2\lambda\sum_{k=0}^{\infty}(\lambda-\mu)^kw^{-(k+1)\rho-1}(1+\beta w^{-\alpha})^{-(k+1)\gamma}\bigg(1-\frac{2(\lambda-\mu)}{w^\rho(1+\beta w^{-\alpha})^\gamma}\bigg)^{-1},\ \bigg|\frac{2(\lambda-\mu)}{w^\rho(1+\beta w^{-\alpha})^\gamma}\bigg|<1,\\
	&=2\lambda\sum_{k=0}^{\infty}\sum_{r=0}^{\infty}2^r(\lambda-\mu)^{k+r}w^{-(k+r+1)\rho-1}(1+\beta w^{-\alpha})^{-(k+r+1)\gamma}.
\end{align*}
Its inverse Laplace transform yields
\begin{equation*}
	m_2(t)=2\lambda\sum_{k=0}^{\infty}\sum_{r=0}^{\infty}2^r(\lambda-\mu)^{k+r}t^{(k+r+1)\rho}E_{\alpha,(k+r+1)\rho+1}^{(k+r+1)\gamma}(-\beta t^\alpha),\ t\ge0.
\end{equation*}
Thus, the variance of GFLBDP is given by
$\mathbb{V}\mathrm{ar}\mathcal{N}_{hp}(t)=m_2(t)+\mathbb{E}\mathcal{N}_{hp}(t)-\mathbb{E}\mathcal{N}_{hp}(t)^2$, $t\ge0$.
\begin{remark}
	For $\gamma=0$, the GFLBDP coincides with the fractional linear birth-death process studied in \cite{Orsingher2011}. So, we have
	\begin{align*}
		\mathbb{E}\mathcal{N}_{hp}(t)|_{\gamma=0}&=\sum_{k=0}^{\infty}(\lambda-\mu)^kt^{k\rho}E_{\alpha,k\rho+1}^{0}(-\beta t^\alpha)\\
		&=\sum_{k=0}^{\infty}\frac{(\lambda-\mu)^kt^{k\rho}}{\Gamma(k\rho+1)}=E_{\rho,1}((\lambda-\mu)t^\rho),\ t\ge0,
	\end{align*}
	which coincides with Eq. (4.3) of \cite{Orsingher2011}.
	
	Also, 
	\begin{align*}
		m_2(t)|_{\gamma=0}&=2\lambda\sum_{k=0}^{\infty}\sum_{r=0}^{\infty}2^r(\lambda-\mu)^{k+r}t^{(k+r+1)\rho}E_{\alpha,(k+r+1)\rho+1}^{0}(-\beta t^\alpha)\\
		&=\frac{\lambda}{\lambda-\mu}\sum_{k=0}^{\infty}\frac{1}{2^k}\sum_{r=0}^{\infty}\frac{(2(\lambda-\mu)t^\rho)^{k+r+1}}{\Gamma((k+r+1)\rho+1)}\\
		&=\frac{\lambda}{\lambda-\mu}\sum_{k=0}^{\infty}\frac{1}{2^k}\sum_{r=k+1}^{\infty}\frac{(2(\lambda-\mu)t^\rho)^{r}}{\Gamma(r\rho+1)}\\
		&=\frac{\lambda}{\lambda-\mu}\sum_{r=1}^{\infty}\frac{(2(\lambda-\mu)t^\rho)^{r}}{\Gamma(r\rho+1)}\sum_{k=0}^{r-1}\frac{1}{2^k}\\
		&=\frac{2\lambda}{\lambda-\mu}\sum_{r=1}^{\infty}\bigg(\frac{(2(\lambda-\mu)t^\rho)^{r}}{\Gamma(r\rho+1)}-\frac{((\lambda-\mu)t^\rho)^{r}}{\Gamma(r\rho+1)}\bigg)\\
		&=\frac{2\lambda}{\lambda-\mu}\big(E_{\rho,1}(2(\lambda-\mu)t^\rho)-E_{\rho,1}((\lambda-\mu)t^\rho)\big),\ t\ge0.
	\end{align*}
	So, the variance of fractional linear birth-death process can be obtained as
	\begin{equation*}
		\mathbb{V}\mathrm{ar}\mathcal{N}_{hp}(t)|_{\gamma=0}=\frac{2\lambda}{\lambda-\mu}E_{\rho,1}(2(\lambda-\mu)t^\rho)-\frac{\lambda+\mu}{\lambda-\mu}E_{\rho,1}((\lambda-\mu)t^\rho)-\Big(E_{\rho,1}((\lambda-\mu)t^\rho)\Big)^2,\ t\ge0,
	\end{equation*}
	which agrees with Eq. (4.11) of \cite{Orsingher2011}.
	
	Further, for $\gamma=0$ and $\rho=1$, we have $\mathbb{E}\mathcal{N}_{hp}(t)|_{\gamma=0,\rho=1}=e^{(\lambda-\mu)t}$ and
	\begin{equation*}
		\mathbb{V}\mathrm{ar}\mathcal{N}_{hp}(t)|_{\gamma=0,\rho=1}=\frac{\lambda+\mu}{\lambda-\mu}e^{(\lambda-\mu)t}(e^{(\lambda-\mu)t}-1),
	\end{equation*}
	which agrees with Eq. (8.48) and Eq. (8.49) of \cite{Bailey1964}, respectively.
\end{remark}

Next,  we derive the explicit expression of the extinction probability of GFLBDP  for three different cases of birth and death rates \textit{viz} $\lambda=\mu$, $\lambda<\mu$ and $\lambda>\mu$. 
First, we consider the case of equal birth and death rate.

From (\ref{tcrep}), the state probabilities of GFLBDP can be written as follows:
\begin{equation}\label{tcpgf}
	p_{hp}(n,t)=\int_{0}^{\infty}p(n,x)\mathrm{Pr}\{\mathscr{Q}(t)\in\mathrm{d}x\},\ n\geq0.
\end{equation}
\begin{theorem}
	For $\lambda=\mu$, the extinction probability of GFLBDP  is given by
	\begin{equation}\label{gfpgf1}
		p_{hp}(0,t)=1-\sum_{k=0}^{\infty}k!(-\lambda t^\rho)^kE_{\alpha,k\rho+1}^{k\gamma}(-\beta t^{\alpha}),
	\end{equation}
	where $E_{\alpha,\beta}^\gamma(\cdot)$ is the three parameter Mittag-Leffler function as defined in (\ref{Mittag3}).
\end{theorem}
\begin{proof}
	By using (\ref{lbdppgf}) in (\ref{tcpgf}), we have
	\begin{equation}\label{pf21}
		p(0,t)=\int_{0}^{\infty}\frac{\lambda x}{1+\lambda x}\mathrm{Pr}\{\mathscr{Q}(t)\in\mathrm{d}x\}.
	\end{equation}
	On taking the Laplace transform on both sides of (\ref{pf21}) and using (\ref{tclap}), we get
	\begin{align*}
		\int_{0}^{\infty}e^{-wt}p_{hp}(0,t)\,\mathrm{d}t
		&=w^{\rho-1}(1+\beta w^{-\alpha})^\gamma\int_{0}^{\infty}\frac{\lambda x}{1+\lambda x}\exp(-w^{\rho}(1+\beta w^{-\alpha})^{\gamma}x)\,\mathrm{d}x\\
		&=w^{\rho-1}(1+\beta w^{-\alpha})^\gamma\int_{0}^{\infty}\bigg(1-\int_{0}^{\infty}e^{-(1+\lambda x)y}\,\mathrm{d}y\bigg)\exp(-w^{\rho}(1+\beta w^{-\alpha})^{\gamma}x)\,\mathrm{d}x\\
		&=w^{\rho-1}(1+\beta w^{-\alpha})^\gamma\bigg(\int_{0}^{\infty}\exp(-w^{\rho}(1+\beta w^{-\alpha})^{\gamma}x)\,\mathrm{d}x\\
		&\ \ -\int_{0}^{\infty}e^{-y}\,\mathrm{d}y\int_{0}^{\infty}\exp(-(w^{\rho}(1+\beta w^{-\alpha})^{\gamma}+\lambda y)x)\,\mathrm{d}x\bigg)\\
		&=\frac{w^{\rho-1}(1+\beta w^{-\alpha})^\gamma}{w^{\rho}(1+\beta w^{-\alpha})^{\gamma}}-\int_{0}^{\infty}\frac{w^{\rho-1}(1+\beta w^{-\alpha})^\gamma e^{-y}}{w^{\rho}(1+\beta w^{-\alpha})^{\gamma}+\lambda y}\,\mathrm{d}y\\
		&=\frac{1}{w}-\frac{1}{w}\int_{0}^{\infty}\bigg(1+\frac{\lambda y}{w^\rho(1+\beta w^{-\alpha})^\gamma}\bigg)^{-1}e^{-y}\,\mathrm{d}y\\
		&=\frac{1}{w}-\sum_{k=0}^{\infty}\int_{0}^{\infty}(-\lambda y)^kw^{-k\rho-1}(1+\beta w^{-\alpha})^{-k\gamma}e^{-y}\,\mathrm{d}y,\ \bigg|\frac{\lambda y}{w^{\rho}(1+\beta w^{-\alpha})^{\gamma}}\bigg|<1,\\
		&=\frac{1}{w}-\sum_{k=0}^{\infty}k!(-\lambda )^kw^{-k\rho-1}(1+\beta w^{-\alpha})^{-k\gamma}
	\end{align*}
and the required result follows from (\ref{mllap}).
\end{proof} 
\begin{remark}
	For $\gamma=0$ in (\ref{gfpgf1}), we get 
	\begin{align}
		p_{hp}(0,t)|_{\gamma=0}&=1-\sum_{k=0}^{\infty}k!(-\lambda t^\rho)^kE_{\alpha,k\rho+1}^{0}(-\beta t^{\alpha})\nonumber\\
		&=1-\sum_{k=0}^{\infty}\int_{0}^{\infty}y^ke^{-y}\,\mathrm{d}y\frac{(-\lambda t^\rho)^k}{\Gamma(k\rho+1)}=1-\int_{0}^{\infty}E_{\rho,1}(-\lambda t^\rho y)e^{-y}\,\mathrm{d}y,\label{fextp1}
	\end{align}
	which coincides with the extinction probability of  fractional linear birth-death process  (see \cite{Orsingher2011}, Eq. (2.26)).
	Further, for $\gamma=0$ and $\rho=1$, by using $E_{\alpha,k+1}^{0}(-\beta t^\alpha)=1/k!$, we get $p_{hp}(0,t)|_{\gamma=0,\rho=1}={\lambda t}/({\lambda t+1})$,
	which coincides with (\ref{lbdppgf}).
\end{remark}
\begin{remark}
	Equivalently, the extinction probability (\ref{gfpgf1}) can be written as
	\begin{equation}\label{extrep1}
		p_{hp}(0,t)=1-\sum_{k=0}^{\infty}k!(-\lambda t^\rho)^kE_{\alpha,k\rho+1}^{k\gamma}(-\beta t^{\alpha})=1-\sum_{k=0}^{\infty}\int_{0}^{\infty}e^{-y}(-\lambda t^\rho y)^k\,\mathrm{d}yE_{\alpha,k\rho+1}^{k\gamma}(-\beta t^{\alpha}).
	\end{equation}	
	 Garra \textit{et al.} \cite{Garra2014} studied a generalized fractional Poisson process with parameter $\lambda>0$ where the regularized Hilfer Prabhakar derivative ${}^{hp}\mathcal{D}_{\alpha, -\beta}^{\gamma,\rho}$ is used. Its inter arrival times $\mathcal{T}^\lambda$ has the following distribution: 
	\begin{equation}\label{iat}
		\mathrm{Pr}\{\mathcal{T}^\lambda>t\}=\sum_{k=0}^{\infty}(-\lambda t^\rho)^kE_{\alpha,k\rho+1}^{k\gamma}(-\beta t^\rho),\ t\ge0.
	\end{equation} 
	
	Let $\mathscr{E}$ be an exponential random variable with parameter $1$ which is independent of $\mathcal{T}^\lambda$. Then, from (\ref{extrep1}), we have
	$
		p_{hp}(0,t)=1-\mathrm{Pr}\{\mathcal{T}^{\lambda \mathscr{E}}>t\}.
$
	That is, the extinction probability of GFLBDP can be expressed in terms of the distribution function of inter arrival times of generalized fractional Poisson process with parameter $\lambda\mathscr{E}$.
\end{remark}
\begin{remark}\label{asextp1}
	From (\ref{mittaglim}), it follows that as $t\to\infty$, the extinction probability (\ref{gfpgf1}) has the following limiting approximation:
	\begin{align*}
		p_{hp}(0,t)&\sim 1-\sum_{k=0}^{\infty}k!(-\lambda t^\rho)^k\frac{(\beta t^\alpha)^{-k\gamma}}{\Gamma(k(\rho-\alpha\gamma)+1)}\\
		&=1-\sum_{k=0}^{\infty}\int_{0}^{\infty}e^{-y}y^k\,\mathrm{d}y(-\lambda t^\rho)^k\frac{(\beta t^\alpha)^{-k\gamma}}{\Gamma(k(\rho-\alpha\gamma)+1)}\\
		&=1-\int_{0}^{\infty}e^{-y}E_{\rho-\alpha\gamma,1}(-\lambda\beta^{-\gamma}y t^{\rho-\alpha\gamma})\,\mathrm{d}y.
	\end{align*}
	So, if $0<\rho-\alpha\gamma<1$ and $\beta>0$, in view of (\ref{fextp1}), we note that the asymptotic extinction probability of GFLBDP for $\lambda=\mu$ case coincides to that of the fractional birth-death process of fractional derivative order $\rho-\alpha\gamma$ and birth rate  $n\lambda\beta^{-\gamma}$. For $\beta=1$, this can be observed directly from Remark \ref{asymprem} and the time-changed representation (\ref{tcrep}).
\end{remark}

Next, we consider the case where the birth rate is lower than the death rate.
\begin{theorem}
For $\lambda<\mu$, the extinction probability of GFLBDP is given by
\begin{equation}\label{extp2}
	p_{hp}(0,t)=1-\bigg(\frac{\mu}{\lambda}-1\bigg)\sum_{k=1}^{\infty}\sum_{r=0}^{\infty}\bigg(\frac{\lambda}{\mu}\bigg)^k(-k(\mu-\lambda)t^\rho)^rE_{\alpha,r\rho+1}^{r\gamma}(-\beta t^\alpha),\ \lambda<\mu,\ t\ge0.
\end{equation}
\end{theorem}
\begin{proof}
	By using (\ref{lbdppgf}) in (\ref{tcpgf}), and on taking the Laplace transform, we have
	\begin{align*}
		\int_{0}^{\infty}&e^{-wt}p_{hp}(0,t)\,\mathrm{d}t\\
		&=w^{\rho-1}(1+\beta w^{-\alpha})^\gamma\int_{0}^{\infty}\frac{\mu-\mu e^{(\lambda-\mu)x}}{\mu-\lambda e^{(\lambda-\mu)x}}\exp(-w^{\rho}(1+\beta w^{-\alpha})^{\gamma}x)\,\mathrm{d}x\\
		&=w^{\rho-1}(1+\beta w^{-\alpha})^\gamma\int_{0}^{\infty}\bigg(1+e^{-(\mu-\lambda)x}\bigg)\bigg(1-\frac{\lambda}{\mu}e^{-(\mu-\lambda)x}\bigg)^{-1}\exp(-w^{\rho}(1+\beta w^{-\alpha})^{\gamma}x)\,\mathrm{d}x\\
		&=w^{\rho-1}(1+\beta w^{-\alpha})^\gamma\sum_{k=0}^{\infty}\bigg(\bigg(\frac{\lambda}{\mu}\bigg)^k\int_{0}^{\infty}\exp(-(k(\mu-\lambda)+w^\rho(1+\beta w^{-\alpha})^\gamma)x)\,\mathrm{d}x\\
		&\ \ -\frac{\mu}{\lambda}\bigg(\frac{\lambda}{\mu}\bigg)^{k+1}\int_{0}^{\infty}\exp(-((k+1)(\mu-\lambda)+w^\rho(1+\beta w^{-\alpha})^\gamma)x)\,\mathrm{d}x\bigg)\\
		&=\frac{1}{w}-\bigg(\frac{\mu}{\lambda}-1\bigg)\sum_{k=1}^{\infty}\bigg(\frac{\lambda}{\mu}\bigg)^k\frac{w^{\rho-1}(1+\beta w^{-\alpha})^\gamma}{k(\mu-\lambda)+w^\rho(1+\beta w^{-\alpha})^\gamma}\\
		&=\frac{1}{w}-\bigg(\frac{\mu}{\lambda}-1\bigg)\sum_{k=1}^{\infty}\bigg(\frac{\lambda}{\mu}\bigg)^k\frac{1}{w}\bigg(1+\frac{k(\mu-\lambda)}{w^\rho(1+\beta w^{-\alpha})^\gamma}\bigg)^{-1}\\
		&=\frac{1}{w}-\bigg(\frac{\mu}{\lambda}-1\bigg)\sum_{k=1}^{\infty}\sum_{r=0}^{\infty}\bigg(\frac{\lambda}{\mu}\bigg)^k(-k(\mu-\lambda))^rw^{-r\rho-1}(1+\beta w^{-\alpha})^{-r\gamma},
	\end{align*}
	where $|{k(\mu-\lambda)}{w^{-\rho}(1+\beta w^{-\alpha})^{-\gamma}}|<1$. Its inverse Laplace transform yields the required result on using (\ref{mllap}).
\end{proof}
\begin{remark}
	For $\gamma=0$, the extinction probability (\ref{extp2}) reduces to
	\begin{align*}
		p_{hp}(0,t)&=1-\bigg(\frac{\mu}{\lambda}-1\bigg)\sum_{k=1}^{\infty}\sum_{r=0}^{\infty}\bigg(\frac{\lambda}{\mu}\bigg)^k\frac{(-k(\mu-\lambda)t^\rho)^r}{\Gamma(r\rho+1)}\\
		&=1-\bigg(\frac{\mu}{\lambda}-1\bigg)\sum_{k=1}^{\infty}\sum_{r=0}^{\infty}\bigg(\frac{\lambda}{\mu}\bigg)^kE_{\rho,1}(-k(\mu-\lambda)t^\rho),\ t\ge0,
	\end{align*}
	which coincides with Eq. (2.20) of \cite{Orsingher2011}.
	Also, for $\gamma=0$ and $\rho=1$, it reduces to
	\begin{align*}
		p_{hp}(0,t)|_{\gamma=0,\rho=1}
		&=\frac{\mu-\mu e^{(\lambda-\mu)t}}{\mu-\lambda e^{(\lambda-\mu)t}},\ t\ge0,
	\end{align*}
	which agrees with (\ref{lbdppgf}).
\end{remark}
\begin{remark}
	Let $\mathcal{Z}$ be a geometric random variable with distribution 
	\begin{equation*}
		\mathrm{Pr}\{\mathcal{Z}=k|\mathcal{Z}=1\}=\bigg(1-\frac{\lambda}{\mu}\bigg)\bigg(\frac{\lambda}{\mu}\bigg)^{k-1},\ k\ge1,
	\end{equation*}
	where $\lambda<\mu$.
	So, the extinction probability (\ref{extp2}) has the following representation:
	\begin{align*}
		p_{hp}(0,t)&=1-\bigg(1-\frac{\lambda}{\mu}\bigg)\sum_{k=1}^{\infty}\sum_{r=0}^{\infty}\bigg(\frac{\lambda}{\mu}\bigg)^{k-1}(-k(\mu-\lambda)t^\rho)^rE_{\alpha,r\rho+1}^{r\gamma}(-\beta t^\alpha)\\
		&=1-\sum_{k=1}^{\infty}\sum_{r=0}^{\infty}\mathrm{Pr}\{\mathcal{Z}=k|\mathcal{Z}=1\}(-k(\mu-\lambda)t^\rho)^rE_{\alpha,r\rho+1}^{r\gamma}(-\beta t^\alpha),\ t\ge0.
	\end{align*}
	Thus, in the case of $\lambda<\mu$, the extinction probability of GFLBDP can be expressed in terms of the distribution function of inter arrival times of generalized fractional Poisson process with parameter $(\mu-\lambda)\mathcal{Z}$.
\end{remark}
\begin{remark}\label{asextp2}
		For large $t$, by using \eqref{mittaglim}, the extinction probability (\ref{extp2}) has the following limiting behaviour:
		\begin{equation}\label{extasymm}
			p_{hp}(0,t)\sim1-\bigg(\frac{\mu}{\lambda}-1\bigg)\sum_{k=1}^{\infty}\bigg(\frac{\lambda}{\mu}\bigg)^kE_{\rho-\alpha \gamma,1}(-k(\mu-\lambda)\beta^{-\gamma}t^{\rho-\alpha\gamma}),
		\end{equation}
		where $\rho>\alpha \gamma$. So, if $0<\rho-\alpha\gamma<1$ and $\beta>0$ then from (\ref{extp2}) and \eqref{extasymm}, it follows that the asymptotic extinction probability of GFLBDP when the birth rate is lower than the death rate coincides to that of the fractional birth-death process of fractional derivative order $\rho-\alpha\gamma$ with birth rate  $n\lambda\beta^{-\gamma}$ and death rate $n\mu\beta^{-\gamma}$, $n\ge0$.
\end{remark}

Finally, we consider the case where the birth rate is higher than the death rate.

\begin{theorem}
	For $\lambda>\mu$, the extinction probabilities of GFLBDP is 
	\begin{equation}\label{pgf3}
		p_{hp}(0,t)=\frac{\mu}{\lambda }-\bigg(1-\frac{\mu}{\lambda}\bigg)\sum_{k=1}^{\infty}\sum_{r=0}^{\infty}\bigg(\frac{\mu}{\lambda}\bigg)^k(-k(\lambda-\mu)t^\rho)^rE_{\alpha,r\rho+1}^{r\gamma}(-\beta t^\alpha),\ t\ge0.
	\end{equation}
\end{theorem}
\begin{proof}
	From (\ref{lbdppgf}) and (\ref{tcpgf}), we have the following Laplace transform:
	\begin{align*}
	 \int_{0}^{\infty}&e^{-wt}p_{hp}(u,t)\,\mathrm{d}t\\
		&=w^{\rho-1}(1+\beta w^{-\alpha})^\gamma\int_{0}^{\infty}\frac{\mu-\mu e^{(\lambda-\mu)x}}{\mu-\lambda e^{(\lambda-\mu)x}}\exp(-w^{\rho}(1+\beta w^{-\alpha})^{\gamma}x)\,\mathrm{d}x\\
		&=w^{\rho-1}(1+\beta w^{-\alpha})^\gamma\frac{\mu}{\lambda}\int_{0}^{\infty}\bigg(1-e^{-(\lambda-\mu)x}\bigg)\bigg(1-\frac{\mu}{\lambda}e^{-(\lambda-\mu)x}\bigg)^{-1}\exp(-w^{\rho}(1+\beta w^{-\alpha})^{\gamma}x)\,\mathrm{d}x\\
		&=w^{\rho-1}(1+\beta w^{-\alpha})^\gamma\frac{\mu}{\lambda}\sum_{k=0}^{\infty}\bigg(\bigg(\frac{\mu}{\lambda}\bigg)^k\int_{0}^{\infty}\exp(-(k(\lambda-\mu)+w^\rho(1+\beta w^{-\alpha})^\gamma)x)\,\mathrm{d}x\\
		&\ \ -\frac{\lambda}{\mu}\bigg(\frac{\mu}{\lambda }\bigg)^{k+1}\int_{0}^{\infty}\exp(-((k+1)(\lambda-\mu)+w^\rho(1+\beta w^{-\alpha})^\gamma)x)\,\mathrm{d}x\bigg)\\
		&=\frac{\mu}{\lambda w}-\bigg(1-\frac{\mu}{\lambda}\bigg)\sum_{k=1}^{\infty}\bigg(\frac{\mu}{\lambda}\bigg)^k\frac{w^{\rho-1}(1+\beta w^{-\alpha})^\gamma}{k(\lambda-\mu)+w^\rho(1+\beta w^{-\alpha})^\gamma}\\
		&=\frac{\mu}{\lambda w}-\bigg(1-\frac{\mu}{\lambda}\bigg)\sum_{k=1}^{\infty}\bigg(\frac{\mu}{\lambda}\bigg)^k\frac{1}{w}\bigg(1+\frac{k(\lambda-\mu)}{w^\rho(1+\beta w^{-\alpha})^\gamma}\bigg)^{-1}\\
		&=\frac{\mu}{\lambda w}-\bigg(1-\frac{\mu}{\lambda}\bigg)\sum_{k=1}^{\infty}\sum_{r=0}^{\infty}\bigg(\frac{\mu}{\lambda}\bigg)^k(-k(\lambda-\mu))^rw^{-r\rho-1}(1+\beta w^{-\alpha})^{-r\gamma},
	\end{align*}
	where $|{k(\lambda-\mu)}{w^{-\rho}(1+\beta w^{-\alpha})^{-\gamma}}|<1$.
	Its inversion yields the required result.
\end{proof}
\begin{remark}
	For $\gamma=0$, the extinction probability (\ref{pgf3}) reduces to
	\begin{align*}
		p_{hp}(0,t)|_{\gamma=0}&=\frac{\mu}{\lambda }-\bigg(1-\frac{\mu}{\lambda}\bigg)\sum_{k=1}^{\infty}\sum_{r=0}^{\infty}\bigg(\frac{\mu}{\lambda}\bigg)^k\frac{(-k(\lambda-\mu)t^\rho)^r}{\Gamma(r\rho+1)}\\
		&=\frac{\mu}{\lambda }-\bigg(1-\frac{\mu}{\lambda}\bigg)\sum_{k=1}^{\infty}\sum_{r=0}^{\infty}\bigg(\frac{\mu}{\lambda}\bigg)^kE_{\rho,1}(-k(\lambda-\mu)t^\rho),
	\end{align*}
	which coincides with Eq. (2.13) of \cite{Orsingher2011}.
	For $\gamma=0$ and $\rho=1$, it reduces to
	\begin{align*}
	p_{hp}(0,t)|_{\gamma=1,\rho=0}
	&=\frac{\mu}{\lambda }-\bigg(1-\frac{\mu}{\lambda}\bigg)\sum_{k=1}^{\infty}\bigg(\frac{\mu}{\lambda}\bigg)^ke^{-k(\lambda-\mu)t}\\
	&=\frac{\mu}{\lambda}-\bigg(1-\frac{\mu}{\lambda}\bigg)\bigg(\frac{1}{1-\frac{\mu}{\lambda}e^{-(\lambda-\mu)t}}-1\bigg)=\frac{-\mu+\mu e^{-(\lambda-\mu)t}}{-\lambda+\mu e^{-(\lambda-\mu)t}},\ t\ge0,
	\end{align*}
	which agrees with (\ref{lbdppgf}).
\end{remark}
\begin{remark}
	Let $\mathscr{Z}$ be a geometric random variable with distribution
	\begin{equation*}
		\mathrm{Pr}\{\mathscr{Z}=k|\mathscr{Z}=1\}=\bigg(1-\frac{\mu}{\lambda}\bigg)\bigg(\frac{\mu}{\lambda}\bigg)^{k-1},\ k\ge1,
	\end{equation*}
	where $\lambda>\mu$.
	So, the extinction probability (\ref{pgf3}) can be written as follows:
	\begin{equation*}
		p_{hp}(0,t)=\frac{\mu}{\lambda }\bigg(1-\sum_{k=1}^{\infty}\sum_{r=0}^{\infty}\mathrm{Pr}\{\mathscr{Z}=k|\mathscr{Z}=1\}(-k(\lambda-\mu)t^\rho)^rE_{\alpha,r\rho+1}^{r\gamma}(-\beta t^\alpha)\bigg),\ t\ge0.
	\end{equation*}
	Thus, for $\lambda>\mu$, the extinction probability of GFLBDP can be expressed in terms of the distribution function of inter arrival times of generalized fractional Poisson process with parameter $(\lambda-\mu)\mathscr{Z}$.
\end{remark}
\begin{remark}\label{asextp3}
		For large $t$, by using \eqref{mittaglim}, the extinction probability (\ref{pgf3}) has the following limiting approximation:
		\begin{equation}\label{extasym}
			p_{hp}(0,t)\sim\frac{\mu}{\lambda }-\bigg(1-\frac{\mu}{\lambda}\bigg)\sum_{k=1}^{\infty}\bigg(\frac{\mu}{\lambda}\bigg)^kE_{\rho-\alpha \gamma,1}(-k(\lambda-\mu)\beta^{-\gamma}t^{\rho-\alpha\gamma}),
		\end{equation}
		where $\rho>\alpha \gamma$. If $0<\rho-\alpha\gamma<1$ and $\beta>0$ then in view of (\ref{pgf3}) and \eqref{extasym},  the asymptotic extinction probability of GFLBDP when birth rate is higher than the death rate coincides to that of the fractional birth-death process of fractional derivative order $\rho-\alpha\gamma$ with birth rate $n\lambda\beta^{-\gamma}$ and death rate $n\mu\beta^{-\gamma}$, $n\ge0$.
\end{remark}
\section{State probabilities of GFLBDP}\label{sec3}
In this section, we  derive the explicit expressions for the state probabilities of GFLBDP $\{\mathcal{N}_{hp}(t)\}_{t\ge0}$ for different cases of birth and death rates.
\begin{theorem}
	For $\lambda=\mu$, the state probabilities $p_{hp}(n,t)=\mathrm{Pr}\{\mathcal{N}_{hp}(t)=n\}$, $n\ge1$ of GFLBDP are given by
	\begin{equation}\label{sp1}
		p_{hp}(n,t)=\frac{(-\lambda)^{n-1}}{n!}\frac{\partial^n}{\partial\lambda^n}\lambda\sum_{r=0}^{\infty}r!(-\lambda t^\rho)^{r}E_{\alpha,r\rho+1}^{r\gamma}(-\beta w^{-\alpha}).
	\end{equation}
\end{theorem}
\begin{proof}
	On using (\ref{bdpsp}) and (\ref{tcrep}), we have
	\begin{equation*}
		p_{hp}(n,t)=\int_{0}^{\infty}\frac{(\lambda x)^{k-1}}{(1+\lambda x)^{k+1}}\mathrm{Pr}\{\mathscr{Q}(t)\in\mathrm{d}x\},\ n\ge1.
	\end{equation*}
	On using (\ref{tclap}), its Laplace transform is  given by
	{\small\begin{align*}
	&\int_{0}^{\infty}e^{-wt}p_{hp}(n,t)\,\mathrm{d}t\\
	&=w^{\rho-1}(1+\beta w^{-\alpha})^\gamma\int_{0}^{\infty}\frac{(\lambda x)^{n-1}}{(1+\lambda x)^{n+1}}\exp(-w^{\rho}(1+\beta w^{-\alpha})^\gamma x)\,\mathrm{d}x\\
	&=w^{\rho-1}(1+\beta w^{-\alpha})^\gamma\frac{(-1)^n\lambda^{n-1}}{n!}\frac{\partial^n}{\partial\lambda^n}\int_{0}^{\infty}\frac{1}{x(1+\lambda x)}\exp(-w^{\rho}(1+\beta w^{-\alpha})^\gamma x)\,\mathrm{d}x\\
	&=w^{\rho-1}(1+\beta w^{-\alpha})^\gamma\frac{(-1)^n\lambda^{n-1}}{n!}\frac{\partial^n}{\partial\lambda^n}\int_{0}^{\infty}\bigg(\frac{1}{x}-\frac{\lambda}{1+\lambda x}\bigg)\exp(-w^{\rho}(1+\beta w^{-\alpha})^\gamma x)\,\mathrm{d}x\\
	&=w^{\rho-1}(1+\beta w^{-\alpha})^\gamma\frac{(-1)^n\lambda^{n-1}}{n!}\frac{\partial^n}{\partial\lambda^n}\int_{0}^{\infty}\int_{0}^{\infty}(e^{-xy}-\lambda e^{-(1+\lambda x)y})\,\mathrm{d}y\exp(-w^{\rho}(1+\beta w^{-\alpha})^\gamma x)\,\mathrm{d}x\\
	&=w^{\rho-1}(1+\beta w^{-\alpha})^\gamma\frac{(-1)^n\lambda^{n-1}}{n!}\frac{\partial^n}{\partial\lambda^n}\bigg(\int_{0}^{\infty}\frac{\,\mathrm{d}y}{w^{\rho}(1+\beta w^{-\alpha})^\gamma+y}-\int_{0}^{\infty}\frac{\lambda e^{-y}\,\mathrm{d}y}{w^{\rho}(1+\beta w^{-\alpha})^\gamma+\lambda y}\bigg)\\
	&=\frac{(-1)^n\lambda^{n-1}}{n!}\frac{\partial^n}{\partial\lambda^n}\frac{1}{w}\bigg(\int_{0}^{\infty}\bigg(1+\frac{y}{w^{\rho}(1+\beta w^{-\alpha})^\gamma}\bigg)^{-1}\,\mathrm{d}y-\int_{0}^{\infty}\lambda e^{-y}\bigg(1+\frac{\lambda y}{w^{\rho}(1+\beta w^{-\alpha})^\gamma}\bigg)^{-1}\,\mathrm{d}y\bigg)\\
	&=\frac{(-1)^n\lambda^{n-1}}{n!}\frac{\partial^n}{\partial\lambda^n}\sum_{r=0}^{\infty}\bigg(\int_{0}^{\infty}(-y)^rw^{-r\rho-1}(1+\beta w^{-\alpha})^{-r\gamma}\,\mathrm{d}y-\lambda\int_{0}^{\infty}e^{-y}(-\lambda y)^rw^{-r\rho-1}(1+\beta w^{-\alpha})^{-r\gamma}\,\mathrm{d}y\bigg),
	\end{align*}}
	where $|\lambda y w^{-\rho}(1+\beta w^{-\alpha})^{-\gamma})|<1$. On taking its inverse Laplace transform and using (\ref{mllap}), we get
	{\small\begin{align*}
		p_{hp}(n,t)&=\frac{(-1)^n\lambda^{n-1}}{n!}\frac{\partial^n}{\partial\lambda^n}\sum_{r=0}^{\infty}\bigg(\int_{0}^{\infty}(-yt^\rho)^rE_{\alpha,r\rho+1}^{r\gamma}(-\beta t^\alpha)\,\mathrm{d}y-\lambda\int_{0}^{\infty}e^{-y}(-\lambda y)^{r}E_{\alpha,r\rho+1}^{r\gamma}(-\beta w^{-\alpha})\,\mathrm{d}y\bigg)\\
		&=\frac{(-1)^n\lambda^{n-1}}{n!}\frac{\partial^n}{\partial\lambda^n}\sum_{r=0}^{\infty}\bigg(-\lambda\int_{0}^{\infty}e^{-y}(-\lambda y)^{r}E_{\alpha,r\rho+1}^{r\gamma}(-\beta w^{-\alpha})\,\mathrm{d}y\bigg)\\
		&=\frac{(-1)^n\lambda^{n-1}}{n!}\frac{\partial^n}{\partial\lambda^n}\sum_{r=0}^{\infty}r!(-\lambda)^{r+1}t^{r\rho}E_{\alpha,r\rho+1}^{r\gamma}(-\beta w^{-\alpha}).
	\end{align*}}
This completes the proof.
\end{proof}
\begin{remark}
	 For $\gamma=0$, the state probabilities (\ref{sp1}) reduce to
	\begin{align*}
		p_{hp}(n,t)|_{\gamma=0}&=\frac{(-\lambda)^{n-1}}{n!}\frac{\partial^n}{\partial\lambda^n}\lambda\sum_{r=0}^{\infty}r!\frac{(-\lambda t^\rho)^{r}}{\Gamma(r\rho+1)}\\
		&=\frac{(-\lambda)^{n-1}}{n!}\frac{\partial^n}{\partial\lambda^n}\lambda\sum_{r=0}^{\infty}\int_{0}^{\infty}e^{-y}y^r\,\mathrm{d}y\frac{(-\lambda t^\rho)^{r}}{\Gamma(r\rho+1)}\\
		&=\frac{(-\lambda)^{n-1}}{n!}\frac{\partial^n}{\partial\lambda^n}\lambda\int_{0}^{\infty}e^{-y}E_{\rho,1}(-\lambda yt^\rho)\,\mathrm{d}y,\ n\ge1,\ t\ge0,
	\end{align*}
	which agree with the state probabilities of fractional linear birth-death process (see \cite{Orsingher2011}, Eq. (3.20)). Further, for $\rho=1$, we have
	\begin{align*}
		p_{hp}(n,t)|_{\gamma=0,\rho=1}&=\frac{(-\lambda)^{n-1}}{n!}\frac{\partial^n}{\partial\lambda^n}\lambda\int_{0}^{\infty}e^{-(1+\lambda t)y}\,\mathrm{d}y\\
		&=\frac{(-\lambda)^{n-1}}{n!}\frac{\partial^n}{\partial\lambda^n}\frac{\lambda}{1+\lambda t}=\frac{(\lambda t)^{n-1}}{(1+\lambda t)^{n+1}},\ n\ge1,
	\end{align*}
	which agree with that of the linear birth-death process as given in (\ref{bdpsp}).
\end{remark}

\begin{theorem}\label{less}
	For $\lambda<\mu$, the state probabilities of GFLBDP are given by
	\begin{align}\label{sp2}
		p_{hp}(n,t)&=\bigg(\frac{\lambda-\mu}{\mu}\bigg)^2\bigg(\frac{\lambda}{\mu}\bigg)^{n-1}\sum_{r=0}^{\infty}\sum_{k=0}^{\infty}\binom{r+n}{r}\bigg(\frac{\lambda}{\mu}\bigg)^r\sum_{m=0}^{n-1}(-1)^m\binom{n-1}{m}\nonumber\\
		&\hspace{4cm}\cdot(-(\mu-\lambda)(r+m+1)t^\rho)^kE_{\alpha,k\rho+1}^{k\gamma}(-\beta t^\alpha),\ n\ge1,\ \ t\ge0.
	\end{align}
\end{theorem}
\begin{proof}
	By using (\ref{bdpsp}) and (\ref{tcrep}), we have
	\begin{equation}
		p_{hp}(n,t)=(\lambda-\mu)^2\int_{0}^{\infty}e^{-(\lambda-\mu)x}\frac{(\lambda-\lambda e^{-(\lambda-\mu)x})^{k-1}}{(\lambda-\mu e^{-(\lambda-\mu)x})^{k+1}}\mathrm{Pr}\{\mathscr{Q}(t)\in\mathrm{d}x\}.
	\end{equation}
	On using (\ref{tclap}), its Laplace transform is given by 
	{\small\begin{align*}
		\int_{0}^{\infty}e^{-wt}p_{hp}(n,t)\,\mathrm{d}t
		&=w^{\rho-1}(1+\beta w^{-\alpha})^\gamma(\lambda-\mu)^2(-\lambda)^{n-1}\\
		&\hspace{0.8cm}\cdot\int_{0}^{\infty}e^{-(\mu-\lambda)x}\frac{(1-e^{-(\mu-\lambda)x})^{n-1}}{(-\mu)^{n+1}(1-\lambda/\mu e^{-(\mu-\lambda)x})^{n+1}}\exp(-w^{\rho}(1+\beta w^{-\alpha})^\gamma x)\,\mathrm{d}x\\
		&=w^{\rho-1}(1+\beta w^{-\alpha})^\gamma\bigg(\frac{\lambda-\mu}{\mu}\bigg)^2\bigg(\frac{\lambda}{\mu}\bigg)^{n-1}\sum_{r=0}^{\infty}\binom{r+n}{r}\bigg(\frac{\lambda}{\mu}\bigg)^r\sum_{m=0}^{n-1}(-1)^m\binom{n-1}{m}\\
		&\hspace{3cm}\cdot\int_{0}^{\infty}\exp(-(w^{\rho}(1+\beta w^{-\alpha})^\gamma+(\mu-\lambda)(r+m+1))x)\,\mathrm{d}x\\
		&=\bigg(\frac{\lambda-\mu}{\mu}\bigg)^2\bigg(\frac{\lambda}{\mu}\bigg)^{n-1}\sum_{r=0}^{\infty}\binom{r+n}{r}\bigg(\frac{\lambda}{\mu}\bigg)^r\sum_{m=0}^{n-1}(-1)^m\binom{n-1}{m}\\
		&\hspace{4cm}\cdot\frac{w^{\rho-1}(1+\beta w^{-\alpha})^\gamma}{w^{\rho}(1+\beta w^{-\alpha})^\gamma+(\mu-\lambda)(r+m+1)}\\
		&=\bigg(\frac{\lambda-\mu}{\mu}\bigg)^2\bigg(\frac{\lambda}{\mu}\bigg)^{n-1}\sum_{r=0}^{\infty}\binom{r+n}{r}\bigg(\frac{\lambda}{\mu}\bigg)^r\sum_{m=0}^{n-1}(-1)^m\binom{n-1}{m}\\
		&\hspace{2.5cm}\cdot\frac{1}{w}\bigg({1+\frac{(\mu-\lambda)(r+m+1)}{w^{\rho}(1+\beta w^{-\alpha})^\gamma}}\bigg)^{-1},\ \bigg|\frac{(\mu-\lambda)(r+m+1)}{w^{\rho}(1+\beta w^{-\alpha})^\gamma}\bigg|<1,\\
		&=\bigg(\frac{\lambda-\mu}{\mu}\bigg)^2\bigg(\frac{\lambda}{\mu}\bigg)^{n-1}\sum_{r=0}^{\infty}\binom{r+n}{r}\bigg(\frac{\lambda}{\mu}\bigg)^r\sum_{m=0}^{n-1}(-1)^m\binom{n-1}{m}\\
		&\hspace{4cm}\cdot\sum_{k=0}^{\infty}(-(\mu-\lambda)(r+m+1))^kw^{-k\rho-1}(1+\beta w^{-\alpha})^{-k\gamma},
	\end{align*}}
	whose inverse Laplace transform gives the required result.
\end{proof}
\begin{remark}
	For $\gamma=0$, the state probabilities (\ref{sp2}) reduce to 
	{\footnotesize\begin{align*}
		p_{hp}(n,t)|_{\gamma=0}&=\bigg(\frac{\lambda-\mu}{\mu}\bigg)^2\bigg(\frac{\lambda}{\mu}\bigg)^{n-1}\sum_{r=0}^{\infty}\sum_{k=0}^{\infty}\binom{r+n}{r}\bigg(\frac{\lambda}{\mu}\bigg)^r\sum_{m=0}^{n-1}(-1)^m\binom{n-1}{m}\frac{(-(\mu-\lambda)(r+m+1)t^\rho)^k}{\Gamma(k\rho+1)}\\
		&=\bigg(\frac{\lambda-\mu}{\mu}\bigg)^2\bigg(\frac{\lambda}{\mu}\bigg)^{n-1}\sum_{r=0}^{\infty}\sum_{k=0}^{\infty}\binom{r+n}{r}\bigg(\frac{\lambda}{\mu}\bigg)^r\sum_{m=0}^{n-1}(-1)^m\binom{n-1}{m}E_{\rho,1}(-(\mu-\lambda)(r+m+1)t^\rho),\ n\ge1,
	\end{align*}}
	which agree with Eq. (3.16) of \cite{Orsingher2011}. Further, for $\rho=1$, these reduce to
	{\small\begin{align*}
		p_{hp}(n,t)_{\gamma=0}&=\bigg(\frac{\lambda-\mu}{\mu}\bigg)^2\bigg(\frac{\lambda}{\mu}\bigg)^{n-1}\sum_{r=0}^{\infty}\sum_{k=0}^{\infty}\binom{r+n}{r}\bigg(\frac{\lambda}{\mu}\bigg)^r\sum_{m=0}^{n-1}(-1)^m\binom{n-1}{m}\exp(-(\mu-\lambda)(r+m+1)t)\\
		&=\bigg(\frac{\lambda-\mu}{\mu}\bigg)^2\bigg(\frac{\lambda}{\mu}\bigg)^{n-1}e^{-(\mu-\lambda)t}\sum_{r=0}^{\infty}\binom{r+n}{r}\bigg(\frac{\lambda}{\mu}\bigg)^re^{-r(\mu-\lambda)t}\sum_{m=0}^{n-1}(-1)^m\binom{n-1}{m}e^{-m(\mu-\lambda)t}\\
		&=\bigg(\frac{\lambda-\mu}{\mu}\bigg)^2\bigg(\frac{\lambda}{\mu}\bigg)^{n-1}e^{-(\mu-\lambda)t}\bigg(1-\frac{\lambda}{\mu}e^{-(\mu-\lambda)t}\bigg)^{-(n+1)}(1-e^{-(\mu-\lambda)t})^{n-1},\ n\ge1,
	\end{align*}}
	which coincide with (\ref{bdpsp}).
\end{remark}

Next result gives the state probabilities of GFLBDP when the birth rate is higher than the death rate. Its proof follows along the similar lines to that of Theorem \ref{less}. Thus, it is omitted.
\begin{theorem}
	For $\lambda>\mu$, the state probabilities of GFLBDP are given by
	\begin{align}\label{sp3}
		p_{hp}(n,t)&=\bigg(1-\frac{\mu}{\lambda}\bigg)^2\sum_{r=0}^{\infty}\sum_{k=0}^{\infty}\binom{r+n}{r}\bigg(\frac{\mu}{\lambda}\bigg)^r\sum_{m=0}^{n-1}(-1)^m\binom{n-1}{m}\nonumber\\
		&\hspace{4cm}\cdot(-(\lambda-\mu)(r+m+1)t^\rho)^kE_{\alpha,k\rho+1}^{k\gamma}(-\beta t^\alpha),\ n\ge1.
	\end{align}
\end{theorem}
\begin{remark}
	For $\gamma=0$, the state probabilities (\ref{sp3}) reduce to
	{\small\begin{align*}
		p_{hp}(n,t)|_{\gamma=0}&=\bigg(1-\frac{\mu}{\lambda}\bigg)^2\sum_{r=0}^{\infty}\sum_{k=0}^{\infty}\binom{r+n}{r}\bigg(\frac{\mu}{\lambda}\bigg)^r\sum_{m=0}^{n-1}(-1)^m\binom{n-1}{m}\frac{(-(\lambda-\mu)(r+m+1)t^\rho)^k}{\Gamma(k\rho+1)}\\
		&=\bigg(1-\frac{\mu}{\lambda}\bigg)^2\sum_{r=0}^{\infty}\binom{r+n}{r}\bigg(\frac{\mu}{\lambda}\bigg)^r\sum_{m=0}^{n-1}(-1)^m\binom{n-1}{m}E_{\rho,1}(-(\lambda-\mu)(r+m+1)t^\rho),\ n\ge1,
	\end{align*}}
	which coincide with Eq. (3.1) of \cite{Orsingher2011}. Also, for $\rho=1$, we have
	\begin{align*}
	p_{hp}(n,t)|_{\gamma=0,\rho=1}
	&=\bigg(1-\frac{\mu}{\lambda}\bigg)^2\sum_{r=0}^{\infty}\binom{r+n}{r}\bigg(\frac{\mu}{\lambda}\bigg)^r\sum_{m=0}^{n-1}(-1)^m\binom{n-1}{m}\exp(-(\lambda-\mu)(r+m+1)t)\\
	&=\bigg(1-\frac{\mu}{\lambda}\bigg)^2e^{-(\lambda-\mu)t}\sum_{r=0}^{\infty}\binom{r+n}{r}\bigg(\frac{\mu}{\lambda}\bigg)^re^{-r(\lambda-\mu)t}\sum_{m=0}^{n-1}(-1)^m\binom{n-1}{m}e^{-m(\lambda-\mu)t}\\
	&=\bigg(1-\frac{\mu}{\lambda}\bigg)^2e^{-(\lambda-\mu)t}(1-e^{-(\lambda-\mu)t})^{n-1}\bigg(1-\frac{\mu}{\lambda}e^{-(\lambda-\mu)t}\bigg)^{-(n+1)},\ n\ge1,
	\end{align*}
	which agree with (\ref{bdpsp}).
	\end{remark}
	
	\begin{remark}
		For $\rho>\alpha\gamma$ and $\beta>0$ such that $\rho-\alpha\gamma<1$, as mentioned in Remark \ref{asextp1}, Remark \ref{asextp2} and Remark \ref{asextp3} for the extinction probability, a similar limiting result holds true for the state probabilities of GFLBDP. It is given as follows:
		\begin{equation}
			p_{hp}(n,t)\sim\begin{cases}
				\frac{(-1)^n\lambda^{n-1}}{n!}\frac{\partial^n}{\partial\lambda^n}\lambda\int_{0}^{\infty}e^{-y}E_{\rho-\alpha\gamma,1}(-\lambda y\beta^{-\gamma}t^{\rho-\alpha\gamma})\,\mathrm{d}y,\ \lambda=\mu,\vspace{0.2cm}\\
				\big(\frac{\lambda-\mu}{\mu}\big)^2\big(\frac{\lambda}{\mu}\big)^{n-1}\sum_{r=0}^{\infty}\binom{r+n}{r}\big(\frac{\lambda}{\mu}\big)^r\sum_{m=0}^{n-1}(-1)^m\binom{n-1}{m}\\
				\hspace{3cm}\cdot E_{\rho-\alpha\gamma,1}(-(\mu-\lambda)(r+m+1)\beta^{-\gamma}t^{\rho-\alpha\gamma}),\ \lambda<\mu,\vspace{0.2cm}\\
				\big(1-\frac{\mu}{\lambda}\big)^2\sum_{r=0}^{\infty}\binom{r+n}{r}\big(\frac{\mu}{\lambda}\big)^r\sum_{m=0}^{n-1}(-1)^m\binom{n-1}{m}\\
				\hspace{3cm}\cdot E_{\rho-\alpha\gamma,1}(-(\lambda-\mu)(r+m+1)\beta^{-\gamma}t^{\rho-\alpha\gamma}),\ \lambda>\mu.
			\end{cases}
		\end{equation}
		Thus, from Eq. (1.13) of \cite{Orsingher2011}, it follows that in the long run of GFLBDP, its limiting distribution coincides to that of the fractional linear birth-death process of fractional derivative order $\rho-\alpha\gamma$ with birth rate $n\lambda\beta^{-\gamma}$ and death rate  $n\mu\beta^{-\gamma}$, $n\ge0$.
	\end{remark}

\section{Integrals of GFLBDP}\label{sec4}
	Puri \cite{Puri1966} studied the joint distribution of linear birth-death process $\{\mathcal{N}(t)\}_{t\ge0}$ and its path integral. McNeil \cite{McNeil1970} derived the explicit form of the distribution function for the integral $\mathcal{Y}_k=\int_{0}^{\mathcal{T}_k}\mathcal{N}(s)\,\mathrm{d}s$
	of the linear birth-death process. Here, $\mathcal{T}_k$ is the first passage time of linear birth-death process to state $n=0$ given that it starts from the state $n=k$. The integral functionals of type $\mathcal{Y}_k$ are relevant in many real-life situations. For example, in inventory or storage systems, it represents the overall cost of storing a commodity until it is either sold or expired. In \cite{Vishwakarma2024b}, some integrals of the fractional linear birth-death process are studied. The study of integrals of random processes is motivated due to their occurrence in various fields of applied mathematics (see \cite{Orsingher2013}, \cite{Vishwakarma2024a}, and reference therein).
	
	Here, we consider the Prabhakar integral of GFLBDP. For $t\ge0$, the Prabhakar integral of the GFLBDP is defined as follows:
	\begin{equation}\label{piN}
		Y_{\alpha',\rho',\beta'}^{\gamma'}(t)=\int_{0}^{t}(t-s)^{\rho'-1}E_{\alpha',\rho'}^{\gamma'}(\beta'(t-s)^{\alpha'})\mathcal{N}_{hp}(s)\,\mathrm{d}s,\ 0<\alpha'\leq1,\ 0<\rho'\leq1,\ \beta'>0,\ \gamma'\ge0.
	\end{equation}
	On using (\ref{meannhp}), its mean is given by
	\begin{align*}
		\mathbb{E}Y_{\alpha',\rho',\beta'}^{\gamma'}(t)&=\sum_{k=0}^{\infty}(\lambda-\mu)^k\int_{0}^{t}(t-s)^{\rho'-1}E_{\alpha',\rho'}^{\gamma'}(\beta'(t-s)^{\alpha'})s^{k\rho}E_{\alpha,k\rho+1}^{k\gamma}(-\beta s^\alpha)\,\mathrm{d}s\\
		&=\sum_{k=0}^{\infty}\sum_{r=0}^{\infty}\sum_{l=0}^{\infty}\frac{(-1)^l(\lambda-\mu)^k\beta'^{r}\beta^{l}(\gamma')_r(k\gamma)_l}{\Gamma(r\alpha'+\rho')\Gamma(l\alpha+k\rho+1)r!l!}\int_{0}^{t}s^{k\rho+\alpha l}(t-s)^{\rho'+r\alpha'-1}\,\mathrm{d}s\\
		&=\sum_{k=0}^{\infty}\sum_{r=0}^{\infty}\sum_{l=0}^{\infty}\frac{(-1)^l(\lambda-\mu)^k\beta'^{r}\beta^{l}(\gamma')_r(k\gamma)_l}{\Gamma(r\alpha'+\rho')\Gamma(l\alpha+k\rho+1)r!l!}t^{k\rho+l\alpha+\rho'+r\alpha'}\int_{0}^{1}u^{k\rho+l\alpha}(1-u)^{\rho'+r\alpha'-1}\,\mathrm{d}u\\
		&=\sum_{k=0}^{\infty}\sum_{r=0}^{\infty}\sum_{l=0}^{\infty}\frac{(-1)^l(\lambda-\mu)^k\beta'^{r}\beta^{l}(\gamma')_r(k\gamma)_l}{\Gamma(k\rho+l\alpha+\rho'+r\alpha'+1)r!l!}t^{k\rho+l\alpha+\rho'+r\alpha'}\\
		&=\sum_{k=0}^{\infty}\sum_{r=0}^{\infty}\frac{(\lambda-\mu)^k\beta'^{r}(\gamma')_r}{r!}t^{k\rho+\rho'+r\alpha'}E_{\alpha,k\rho+\rho'+r\alpha'+1}^{k\gamma}(-\beta t^\alpha),\ t\ge0.
	\end{align*}
\begin{remark}
	For $\gamma=0$, the integral (\ref{piN}) reduces to the Prabhakar integral of the fractional linear birth-death process whose mean is given by
	\begin{align}
		\mathbb{E}Y_{\alpha',\rho',\beta'}^{\gamma'}(t)|_{\gamma=0}&=\sum_{k=0}^{\infty}\sum_{r=0}^{\infty}\frac{(\lambda-\mu)^k\beta'^{r}(\gamma')_rt^{k\rho+\rho'+r\alpha'}}{r!\Gamma(k\rho+\rho'+r\alpha'+1)}\nonumber\\
		&=\sum_{k=0}^{\infty}(\lambda-\mu)^kt^{k\rho+\rho'}E_{\alpha',k\rho+\rho'+1}^{\gamma'}(\beta't^{\alpha'}),\ t\ge0.\label{meanintg}
	\end{align}
	For $\gamma'=0$, the mean \eqref{meanintg} reduces to the Riemann-Liouville fractional integral of the fractional linear birth-death process (see \cite{Vishwakarma2024b}).
	\end{remark}
	
	For $\gamma=\gamma'=0$ and $\rho=\rho'=1$, the integral (\ref{piN}) reduces to the path integral of the homogeneous linear birth-death process. It is given by $\mathcal{Y}(t)=\int_{0}^{t}\mathcal{N}(s)\,\mathrm{d}s$. The joint distribution $p(n,x,t)=\mathrm{Pr}\{\mathcal{N}(t)=n,\mathcal{Y}(t)\leq x\}$, $n\ge0$, $x\ge0$ of bivariate random process $\{(\mathcal{N}(t),\mathcal{Y}(t))\}_{t\ge0}$ solves the following system of partial differential equations (see \cite{McNeil1971}, Eq. (2.1)):
	\begin{equation*}
		\partial_tp(n,x,t)+n\partial_xp(n,x,t)=-n(\lambda+\mu)p(n,x,t)+(n-1)\lambda p(n-1,x,t)+(n+1)\mu p(n+1,x,t),
	\end{equation*}
	with initial condition $p(1,0,0)=1$. 
	
	Let $r_1(v)$ and $r_2(v)$ be the roots of $\lambda u^2+(iv-\lambda-\mu)u+\mu=0$. Then, the characteristic function $\Phi(u,v,t)=\mathbb{E}\exp(iu\mathcal{N}(t)+iv\mathcal{Y}(t))$, $u\in\mathbb{R}$, $v\in\mathbb{R}$ is given by (see \cite{Puri1966}, Eq. (15))
	\begin{equation}\label{jch}
	\Phi(u,v,t)=\bigg(r_2(v)+\frac{r_1(v)-r_2(v)}{1-(e^{iu}-r_1(v))(e^{iu}-r_2(v))^{-1}e^{\lambda(r_1(v)-r_2(v))t}}\bigg),\ t\ge0,
	\end{equation}
where  $r_1(v)=(2\lambda)^{-1}({\lambda+\mu-iv+\sqrt{(\lambda+\mu-iv)^2-4\lambda\mu}})$ and  $r_2(v)=(2\lambda)^{-1}(\lambda+\mu-iv - $ $\sqrt{(\lambda+\mu-iv)^2-4\lambda\mu})$.
	
	Let us consider a bivariate random process $\{(\mathcal{N}_{hp}(t),\mathcal{Y}_{hp}(t))\}_{t\ge0}$ whose distribution $p_{hp}(n,x,t)$ $=\mathrm{Pr}\{\mathcal{N}_{hp}(t)=n,\mathcal{Y}(t)\leq x\}$, $n\ge0$, $x\ge0$ solves the following system of fractional differential equations:
	\begin{equation*}
		{}^{hp}\mathcal{D}_{\alpha, -\beta}^{\gamma,\rho}p(n,x,t)+n\partial_xp(n,x,t)=-n(\lambda+\mu)p(n,x,t)+(n-1)\lambda p(n-1,x,t)+(n+1)\mu p(n+1,x,t),
	\end{equation*} 
	where ${}^{hp}\mathcal{D}_{\alpha, -\beta}^{\gamma,\rho}$ is the regularized fractional derivative as defined in (\ref{fder}).
	
Let $\{\mathscr{Q}(t)\}_{t\ge0}$ be a random process whose density solves (\ref{tc}). Then, on following the similar lines of the proof of Theorem \ref{thmtc}, it can be established that
	\begin{equation}\label{eqdmj}
	(\mathcal{N}_{hp}(t),\mathcal{Y}_{hp}(t))\overset{d}{=}(\mathcal{N}(\mathscr{Q}(t)),\mathcal{Y}(\mathscr{Q}(t))),
	\end{equation} 
where $\{\mathscr{Q}(t)\}_{t\ge0}$ is independent of the 	bivariate process $\{(\mathcal{N}(t),\mathcal{Y}(t))\}_{t\ge0}$. So, $\{(\mathcal{N}_{hp}(t),\mathcal{Y}_{hp}(t))\}_{t\ge0}$ is a time-changed variant of $\{(\mathcal{N}(t),\mathcal{Y}(t))\}_{t\ge0}$.

\begin{theorem}
The characteristic function $\Phi_{hp}(u,v,t)=\mathbb{E}\exp(iu\mathcal{N}_{hp}(t)+iv\mathcal{Y}_{hp}(t))$, $u\in\mathbb{R}$, $v\in\mathbb{R}$ is given by
{\small\begin{equation}\label{fjch}
\Phi_{hp}(u,v,t)=r_2(v)-(r_1(v)-r_2(v))\sum_{k=0}^{\infty}\sum_{l=0}^{\infty}\bigg(\frac{e^{iu}-r_2(v)}{e^{iu}-r_1(v)}\bigg)^{k+1}(\lambda(r_1(v)-r_2(v))(k+1)t^\rho)^lE_{\alpha,l\rho+1}^{l\gamma}(-\beta t^\alpha).
\end{equation}}
\end{theorem}
\begin{proof}
	 From (\ref{jch}) and (\ref{eqdmj}), we have
	 \begin{equation*}
	 	\Phi_{hp}(u,v,t)=\int_{0}^{\infty}\bigg(r_2(v)+\frac{r_1(v)-r_2(v)}{1-(e^{iu}-r_1(v))(e^{iu}-r_2(v))^{-1}e^{\lambda(r_1(v)-r_2(v))x}}\bigg)\mathrm{Pr}\{\mathscr{Q}(t)\in\mathrm{d}x\}.
	 \end{equation*}
	 On using (\ref{tclap}), its Laplace transform is
	{\small\begin{align}
		\int_{0}^{\infty}&e^{-wt}\Phi_{hp}(u,v,t)\,\mathrm{d}t\nonumber\\
		&=w^{\rho-1}(1+\beta w^{-\alpha})^\gamma\int_{0}^{\infty}\bigg(r_2(v)+\frac{r_1(v)-r_2(v)}{1-(e^{iu}-r_1(v))(e^{iu}-r_2(v))^{-1}e^{\lambda(r_1(v)-r_2(v))x}}\bigg)\nonumber\\
		&\hspace{9.5cm}\cdot\exp(-w^{\rho}(1+\beta w^{-\alpha})^\gamma x)\,\mathrm{d}x\nonumber\\
		&=\frac{r_2(v)}{w}-(r_1(v)-r_2(v))\frac{e^{iu}-r_2(v)}{e^{iu}-r_1(v)}w^{\rho-1}(1+\beta w^{-\alpha})^\gamma\nonumber\\
		&\hspace{1cm}\cdot\int_{0}^{\infty}e^{-\lambda(r_1(v)-r_2(v))s}\bigg(1-\frac{e^{iu}-r_2(v)}{e^{iu}-r_1(v)}e^{-\lambda(r_1(v)-r_2(v))x}\bigg)^{-1}\exp(-w^{\rho}(1+\beta w^{-\alpha})^\gamma x)\,\mathrm{d}x\nonumber\\
		&=\frac{r_2(v)}{w}-(r_1(v)-r_2(v))\sum_{k=0}^{\infty}\bigg(\frac{e^{iu}-r_2(v)}{e^{iu}-r_1(v)}\bigg)^{k+1}w^{\rho-1}(1+\beta w^{-\alpha})^\gamma\nonumber\\
		&\hspace{7.6cm}\cdot\int_{0}^{\infty}e^{-(w^{\rho}(1+\beta w^{-\alpha})^\gamma+\lambda(r_1(v)-r_2(v))(k+1))x}\,\mathrm{d}x\nonumber\\
		&=\frac{r_2(v)}{w}-(r_1(v)-r_2(v))\sum_{k=0}^{\infty}\bigg(\frac{e^{iu}-r_2(v)}{e^{iu}-r_1(v)}\bigg)^{k+1}\frac{w^{\rho-1}(1+\beta w^{-\alpha})^\gamma}{w^{\rho}(1+\beta w^{-\alpha})^\gamma+\lambda(r_1(v)-r_2(v))(k+1)}\nonumber\\
		&=\frac{r_2(v)}{w}-(r_1(v)-r_2(v))\sum_{k=0}^{\infty}\bigg(\frac{e^{iu}-r_2(v)}{e^{iu}-r_1(v)}\bigg)^{k+1}\sum_{r=0}^{\infty}(\lambda(r_1(v)-r_2(v))(k+1))^rw^{r\rho-1}(1+\beta w^{-\alpha})^{-r\gamma},\label{chflap}
	\end{align}}
	where $|(\lambda(r_1(v)-r_2(v))(k+1))w^{-\rho}(1+\beta w^{-\alpha})^{-\gamma}|<1$, $w>0$. On taking the inverse Laplace transform in (\ref{chflap}) and using (\ref{mllap}), we get the required result.
\end{proof}

\begin{remark}
	From (\ref{mittaglim}), for large $t$, the joint characteristics function (\ref{fjch}) has the following limiting approximation:
{\small\begin{align*}
			\Phi_{hp}(u,v,t)&\sim r_2(v)-(r_1(v)-r_2(v))\sum_{k=0}^{\infty}\sum_{l=0}^{\infty}\bigg(\frac{e^{iu}-r_2(v)}{e^{iu}-r_1(v)}\bigg)^{k+1}(\lambda(r_1(v)-r_2(v))(k+1)t^\rho)^l\frac{(\beta t^\alpha)^{-l\gamma}}{\Gamma(l\rho+1-\alpha l\gamma)}\nonumber\\
			&=r_2(v)-(r_1(v)-r_2(v))\sum_{k=1}^{\infty}\bigg(\frac{e^{iu}-r_2(v)}{e^{iu}-r_1(v)}\bigg)^{k}E_{\rho-\alpha\gamma,1}(\lambda (r_1(v)-r_2(v))kt^{\rho-\alpha\gamma}\beta^{-\gamma}),\label{jocharny}
\end{align*}}
which is the joint characteristics function of a time-changed bivariate process (see \cite{Vishwakarma2024b}, Eq. (4.1)). In view of Remark 4.1 of \cite{Vishwakarma2024b}, it follows that if the birth rate is less than or equal to the death rate then the integral process  $\{\mathcal{Y}_{hp}(t)\}_{t\ge0}$ has stable limiting distribution as $t \to \infty$, and its limiting characteristics function is a linear combination of the characteristic functions of infinitely many scaled chi-square random variables. Also, if the birth rate is higher than the death rate then it converges to a non-finite random variable which takes infinity value with probability $1-\mu/\lambda$.
\end{remark}	
\subsection{Application in a genetic model} In \cite{Moran1964}, a stochastic logistic models is used to study biological population that is confined between two limits, where the death rate increases and birth rate decreases linearly with population size. In \cite{Prendiville1949}, a modified model is suggested which reduces to the Ehrenfest model by shifting the origin. 

Here, we discuss an application of the integral of a linear birth-death process at random time. It will be used to study a genetic model that has an upper bound for the population size.
Let us consider a population consisting of $M= 2K$ haploid individuals, each genetically classified as either type $H$ or type $h$. In this model, the total number of individuals $n=0,1,\dots,M$ of type $H$ represent the state of system. At any instance, a randomly selected haploid individual dies and is replaced by a new individual of type $H$ or type $h$. Also, the state dependent birth and death rates are $\lambda_n=(M-n)\lambda$ and $\mu_n=n\mu$, respectively, where $\lambda$ and $\mu$ are positive constants. Let $\mathscr{N}(t)$ denote the number of individuals of type $H$ at time $t$ with $\mathscr{N}(0)=n_0$, $0< n_0< M$. Then, its path integral $\mathscr{Y}(t)=\int_{0}^{t}\mathscr{N}(s)\,\mathrm{d}s$ represent the total number of individuals of type $H$ up to time $t$.

Let us consider the time-changed bivariate process $(\mathscr{N}_\rho(t),\mathscr{Y}_\rho(t))\coloneqq(\mathscr{N}(\mathscr{Q}(t)),\mathscr{Y}(\mathscr{Q}(t)))|_{\gamma=0}$, $0<\rho\leq1$, $t\ge0$, where $\{\mathscr{Q}(t)\}_{t\ge0}$ is a random process whose density solves (\ref{tc}) and it is independent of $\{(\mathscr{N}(t),\mathscr{Y}(t))\}_{t\ge0}$. Thus, the process $\{(\mathscr{N}_\rho(t),\mathscr{Y}_\rho(t))\}_{t\ge0}$ is a potential mathematical model to study the rapid change in a genetic population in which individuals are haploid. On using Eq. (6.2) of \cite{McNeil1971}, from Theorem \ref{thmtc}, it follows that the joint distribution $p_\rho(n,x,t)=\mathrm{Pr}\{\mathscr{N}(t)=n,\mathscr{Y}(t)\leq x\}$, $x\ge0$, $0\leq n\leq M$ solves
\begin{multline*}
	{}^c\mathcal{D}_t^\rho p_\rho(n,x,t)+n\partial_x p_\rho(n,x,t)=-((M-n)\lambda+n\mu)p_\rho(n,x,t)\\
	+(M-n+1)\lambda p_\rho(n,x,t)+(n+1)\mu p_\rho(n,x,t),
\end{multline*}
with $p_\rho(n_0,0,0)=1$. Here, ${}^c\mathcal{D}_t^\rho$ is the Caputo fractional derivative defined as follows (see \cite{Kilbas2006}):
\begin{equation*}
	{}^c\mathcal{D}_t^\rho f(t)=\begin{cases}
		\frac{1}{\Gamma(1-\rho)}\int_{0}^{t}(t-s)^{-\rho}(\mathrm{d}_sf(s))\,\mathrm{d}s,\ 0<\rho<1,\\
		\mathrm{d}_tf(t),\ \rho=1.
	\end{cases}
\end{equation*}

The state probabilities $p_\rho(n,t)=\lim_{x\to\infty}p_\rho(n,x,t)$ of $\{\mathscr{N}_\rho(t)\}_{t\ge0}$ solve the following system of fractional differential equations:
\begin{equation*}
	{}^c\mathcal{D}_t^\rho p_\rho(n,t)=-((M-n)\lambda+n\mu)p_\rho(n,t)\\
	+(M-n+1)\lambda p_\rho(n-1,t)+(n+1)\mu p_\rho(n,x,t),\ 0\leq n\leq M,
\end{equation*}
with $p_\rho(n_0,0)=1$. As $0$ and $M$ are absorbing states, we assume that $p_\rho(n,t)=0$ for all $n<0$ and $n>M$. So, its pgf $\mathscr{G}_\rho(u,t)=\sum_{n=0}^{M}u^np_\rho(n,t)$, $|u|\leq1$ solves
\begin{equation}\label{apgf}
	{}^c\mathcal{D}_t^\rho\mathscr{G}_\rho(u,t)=-(\lambda u+\mu)(u-1)\partial_u\mathscr{G}_\rho(u,t)+ M\lambda(u-1)\mathscr{G}_\rho(u,t),
\end{equation}
with $\mathscr{G}_\rho(u,0)=u^{n_0}$.

On taking the derivative on both sides of (\ref{apgf}) with respect to $u$ and substituting $u=1$, we have the following  governing differential equation: 
\begin{equation*}
	{}^c\mathcal{D}_t^\rho\mathbb{E}\mathscr{N}_\rho(t)=-(\lambda+\mu)\mathbb{E}\mathscr{N}_\rho(t)+M\lambda
\end{equation*} 
with $\mathbb{E}\mathscr{N}(0)=n_0$. On using (\ref{fderlap}) for $\gamma=0$, we get the following Laplace transform: 
\begin{align*}
	\int_{0}^{\infty}e^{-wt}\mathbb{E}\mathscr{N}_\rho(t)\,\mathrm{d}t&=\frac{w^{\rho-1}n_0}{w^\rho+\lambda+\mu}+\frac{M\lambda}{w(w^\rho+\lambda+\mu)},\ w>0,\\
	&=\frac{w^{\rho-1}n_0}{w^\rho+\lambda+\mu}+\frac{M\lambda}{\lambda+\mu}\bigg(\frac{1}{w}-\frac{w^{\rho-1}}{w^\rho+\lambda+\mu}\bigg).
\end{align*}
Its inversion yields 
\begin{equation}\label{amean}
	\mathbb{E}\mathscr{N}_\rho(t)=n_0E_{\rho,1}(-(\lambda+\mu)t^\rho)+\frac{M\lambda}{\lambda+\mu}(1-E_{\rho,1}(-(\lambda+\mu)t^\rho)),\ t\ge0.
\end{equation}

We note that the marginal process $\{\mathscr{N}_\rho(t)\}_{t\ge0}$, $0<\rho\leq1$ can be used to represent the population of individuals of type $H$ under accelerating conditions. Also, its path integral $\mathscr{X}_\rho(t)=\int_{0}^{t}\mathscr{N}_\rho(s)\,\mathrm{d}s$ represents the total number of individuals of type $H$ up to time $t$. Here, we are interested in the average number of individuals of type $H$ that have existed in the population up to time $t$, and it is given by $\mathscr{A}_\rho(t)=t^{-1}\mathbb{E}\mathscr{X}_\rho(t)$, $t>0$. 

From (\ref{amean}), we have
\begin{align*}
	\mathbb{E}\mathscr{X}_\rho(t)=\int_{0}^{t}\mathbb{E}\mathscr{N}_\rho(s)\,\mathrm{d}s&=\bigg(n_0-\frac{M\lambda}{\lambda+\mu}\bigg)\int_{0}^{t}E_{\rho,1}(-(\lambda+\mu)s^\rho)\,\mathrm{d}s+\frac{M\lambda t}{\lambda+\mu}\\
	&=\bigg(n_0-\frac{M\lambda}{\lambda+\mu}\bigg)\sum_{k=0}^{\infty}\frac{(-1)^k(\lambda+\mu)^k}{\Gamma(k\rho+1)}\int_{0}^{t}s^{k\rho}\,\mathrm{d}s+\frac{M\lambda t}{\lambda+\mu}\\
	&=\bigg(n_0-\frac{M\lambda}{\lambda+\mu}\bigg)\sum_{k=0}^{\infty}\frac{(-1)^k(\lambda+\mu)^kt^{k\rho+1}}{(k\rho+1)\Gamma(k\rho+1)}+\frac{M\lambda t}{\lambda+\mu}\\
	&=\bigg(n_0-\frac{M\lambda}{\lambda+\mu}\bigg)tE_{\rho,2}(-(\lambda+\mu)t^\rho)+\frac{M\lambda t}{\lambda+\mu},\ t>0.
\end{align*}
Thus, 
\begin{equation*}
	\mathscr{A}_\rho(t)=\bigg(n_0-\frac{M\lambda}{\lambda+\mu}\bigg)E_{\rho,2}(-(\lambda+\mu)t^\rho)+\frac{M\lambda }{\lambda+\mu},\ t>0.
\end{equation*}
For sufficiently large $t$, from (\ref{mittaglim}), we have the following limiting approximation of it:
\begin{equation*}
	\mathscr{A}_\rho(t)\sim\bigg(n_0-\frac{M\lambda}{\lambda+\mu}\bigg)\frac{1}{\Gamma(2-\rho)(\lambda+\mu)t^\rho}+\frac{M\lambda }{\lambda+\mu}.
\end{equation*}
\begin{remark}
	For $\rho=1$, the process $\{\mathscr{N}_\rho(t)\}_{t\ge0}$ reduces to the process $\{\mathscr{N}(t)\}_{t\ge0}$. Its mean is given by $\mathbb{E}\mathscr{N}(t)=$ $n_0e^{-(\lambda+\mu)t}+M\lambda(\lambda+\mu)^{-1}(1-e^{(\lambda+\mu)t})$. Hence, the mean of $\mathscr{Y}_\rho(t)$ is 
	\begin{align*}
		\mathbb{E}\mathscr{Y}_\rho(t)&=\int_{0}^{\infty}\mathbb{E}\mathscr{Y}(s)\mathrm{Pr}\{\mathscr{Q}(t)\in\mathrm{d}s\}|_{\gamma=0}\\
		&=\int_{0}^{\infty}\int_{0}^{s}\mathbb{E}\mathscr{N}(x)\,\mathrm{d}x\mathrm{Pr}\{\mathscr{Q}(t)\in\mathrm{d}s\}|_{\gamma=0}\\
		&=\int_{0}^{\infty}\int_{0}^{s}\big(n_0e^{-(\lambda+\mu)x}+M\lambda(\lambda+\mu)^{-1}(1-e^{-(\lambda+\mu)x})\big)\,\mathrm{d}x\mathrm{Pr}\{\mathscr{Q}(t)\in\mathrm{d}s\}|_{\gamma=0}.
	\end{align*}
	On using (\ref{tclap}) for $\gamma=0$, we have 
	\begin{align*}
		\int_{0}^{\infty}e^{-wt}\mathbb{E}\mathscr{Y}_\rho(t)\,\mathrm{d}t&=w^{\rho-1}\int_{0}^{\infty}\int_{0}^{s}\bigg(n_0e^{-(\lambda+\mu)x}+\frac{M\lambda}{\lambda+\mu}(1-e^{-(\lambda+\mu)x})\bigg)\,\mathrm{d}x e^{-w^\rho s}\,\mathrm{d}s,\ w>0,\\
		&=w^{\rho-1}\int_{0}^{\infty}\bigg(\bigg(n_0-\frac{M\lambda}{\lambda+\mu}\bigg)\frac{1-e^{-(\lambda+\mu)s}}{\lambda+\mu}+\frac{M\lambda s}{\lambda+\mu}\bigg)e^{-w^\rho s}\,\mathrm{d}s\\
		&=\frac{1}{\lambda+\mu}\bigg(n_0-\frac{M\lambda}{\lambda+\mu}\bigg)\bigg(\frac{1}{w}-\frac{w^{\rho-1}}{w^\rho+\lambda+\mu}\bigg)+\frac{M\lambda}{(\lambda+\mu)w^{\rho+1}}.
	\end{align*}
	Its inverse Laplace transform gives
	\begin{equation*}
		\mathbb{E}\mathscr{Y}_\rho(t)=\frac{1}{\lambda+\mu}\bigg(n_0-\frac{M\lambda}{\lambda+\mu}\bigg)(1-E_{\rho,1}(-(\lambda+\mu)t^\rho))+\frac{M\lambda t^{\rho}}{(\lambda+\mu)\Gamma(\rho+1)},\ t\ge0.
	\end{equation*}
\end{remark}
\section*{Acknowledgement}
The first author thanks Government of India for the grant of Prime Minister's Research Fellowship, ID 1003066.

\end{document}